\documentclass[12pt,a4paper]{amsart}
\usepackage{a4}
\usepackage{amsfonts, amssymb}
\usepackage[all]{xy}

\usepackage[hidelinks]{hyperref}          
\usepackage{todo}

\newtheorem{thm}{Theorem}

\newtheorem{lem}{Lemma}

\newtheorem{conj}{Conjecture}
\newtheorem{defn}{Definition}

\newtheorem{prop}{Proposition}
\newcommand{\E}{\ensuremath{\mathbb E}}

\numberwithin{equation}{section}
\numberwithin{lem}{section} \numberwithin{problem}{section}
\numberwithin{claim}{section} \numberwithin{defn}{section}

\newcommand{\ex}{\text{ex}}

\newcommand{\Ll}{\mathcal{L}}

\newcommand{\Z}{\mathbb{Z}}

\newcommand{\F}{\mathbb F}

\renewcommand{\Pr}[1]{\mathbb{P}\left (  #1 \right )}
\newcommand{\Esp}[1]{\mathbb{E}\left (  #1 \right )}

\newcommand{\ud}{\,\mathrm{d}}
\begin{document}
\title[On extremal problems in $\Ll$-free sets]{On sets free of sumsets with summands of prescribed size}
\begin{abstract}
We study extremal problems about sets of integers that do not contain sumsets with summands of prescribed size. We analyse both finite sets and infinite sequences. We also study the connections of these problems with extremal problems of graphs and hypergraphs.
\end{abstract}
\author{Javier Cilleruelo}
\author{Rafael Tesoro}
\address{Javier Cilleruelo\\ Instituto de Ciencias Matem\'aticas (CSIC-UAM-UC3M-UCM) and
Departamento de Matem\'aticas\\
Universidad Aut\'onoma de Madrid\\
28049, Madrid, Espa\~na}
\address{Rafael Tesoro\\ Departamento de Matem\'aticas\\
Universidad Aut\'onoma de Madrid\\
28049, Madrid, Espa\~na}

\maketitle

\section{Introduction}
A popular topic in combinatorial/additive number theory is the study of extremal sets of integers free of subsets with some given particular shape.
We tackle here extremal problems about sets that do not contain sumsets with summands of prescribed size, and we show their relationship with extremal problems on graphs that are free of complete $r$-partite subgraphs.

\begin{defn}\label{def} Let $r, \ell_1,\ldots,\ell_r$ be integers with $r\ge 1$ and $2\le \ell_1\le \cdots \le \ell_r$. Given an abelian group $G$ we say that  $A\subset G$ is a $\mathcal L^{(r)}_{\ell_1,\dots,\ell_r}$-free set if $A$ does not contain any sumset of the form  
$$L_1+\cdots+L_r=\{\lambda_1+\cdots+\lambda_r:\ \lambda_i\in L_i, \ i=1,\dots, r\},$$ with  $|L_i|=\ell_i,\ i=1,\dots,r$.
For $r=2$ we simply write $\mathcal L_{\ell_1,\ell_2}$. 
\end{defn}
The degenerate case $r=1$, that we denote by $\mathcal L_{\ell_1},$  is trivial: a set $A$ is $\mathcal L_{\ell_1}$-free $\iff |A| \le \ell_1-1.$



\subsection{$\mathcal L$-free sets problems in intervals and finite abelian groups}
To motivate Definition \ref{def} and the results in this work we start by summarizing the state of knowledge for some particular cases already studied in the literature.

\begin{enumerate}
  \item [i)]\emph{$\Ll_{2,2}$-free sets.} They are just the Sidon sets, those having the property that all the differences $a-a'\; ( a,a'\in A,\ a\ne a')$ are distinct.
Indeed take $L_1=\{a_1,b_1\},\ L_2=\{a_2,b_2\}, \: (a_i 
< b_i)$, then the shape of the sumset $L_1+L_2$ can be depicted as one $2$-point set plus one of its translates:

\medskip

$\qquad \qquad \qquad \qquad \bullet$\textbf{{-----}}$\bullet \qquad \qquad \qquad \qquad \bullet$\textbf{{-----}}$\bullet$

{\tiny $\qquad \qquad \qquad  \qquad {\tiny a_1+a_2}\qquad {\tiny b_1+a_2}\qquad \ \  \qquad \qquad {\tiny a_1+b_2}\qquad {\tiny b_1+b_2}\qquad $}

\medskip

A Sidon set can be characterized as being free of this shape.
      
      \

  \item [ii)]\emph{$\Ll_{2,\ell}$-free sets.} A $\Ll_{2,\ell}$-free set $A$ is characterized by the property that there are no more than $\ell-1$ different ways to express any non-zero element in the ambient group as a difference of two elements of $A$. They have been called $B_2^{\circ}[\ell-1]$ sets \cite{Le} and $B_2^{-}[\ell-1]$ sets \cite{GTV}.

For example the typical shape of a sumset $L_1+L_2$ with $|L_1|=2$ and $|L_2|=3$  is one $2$-point set plus two translates of it:

      \medskip

        $\qquad \bullet$---$\bullet \qquad \qquad \bullet$---$\bullet\qquad \qquad \qquad \qquad \bullet$---$\bullet$

           \medskip
           
The $\Ll_{2,3}$-free sets are characterized as being free of this shape.

\

  \item [iii)]\emph{$\mathcal L_{\ell_1,\ell_2}$-free sets.} The sets that are free of $\ell_1$ translations of sets with $\ell_2$ elements were introduced by Erd\H{o}s and Harzheim \cite{E-H} and have been further studied in \cite{PTT}. For example the $\mathcal L_{3,4}$-free sets are characterized by avoiding the following shape:

      \medskip

   $\bullet$--$\bullet$- - -$\bullet \qquad \qquad\bullet$--$\bullet$- - -$\bullet\qquad \qquad \qquad \bullet$--$\bullet$- - -$\bullet$ $\quad $ $\bullet$--$\bullet$- - -$\bullet$

           \

  \item [iv)]\emph{$\mathcal L^{_{(r)}}_{2,\dots,2}$-free sets.} A Hilbert cube of dimension $r$ is a sumset of the form $L_1+\cdots +L_r$ with $|L_1|=\cdots=|L_r|=2$. Thus  $\mathcal L^{_{(r)}}_{2,\dots,2}$-free sets are those free of Hilbert cubes of dimension $r$. A Hilbert cube of dimension $3$ has this shape:

      \medskip

      $\qquad \bullet$--$\bullet$- - -$\bullet$--$\bullet \qquad \qquad \qquad \bullet$--$\bullet $- - -$  \bullet$--$\bullet$
\end{enumerate}


\ 

Estimating the largest size of a set $A\subset \{1,\dots,n\}$ that is $\mathcal L^{_{(r)}}_{\ell_1,\dots,\ell_r}$-free is an interesting and significant problem.

\smallskip

\begin{defn}
We will denote by $F\big(n,\Ll^{_{(r)}}_{\ell_1,\dots,\ell_r}\big )$ the size of a largest $\Ll^{_{(r)}}_{\ell_1,\dots,\ell_r}$-free set in the interval $\{1,\dots,n\}$.
%
\end{defn}
Our first result is a general upper bound that recovers known upper bounds for  the particular cases considered above.
\begin{thm}\label{thm:main-ubound}For any $r\ge 2$ and $ 2\le \ell_1\le \cdots \le \ell_r$ we have
$$\text{F}\big (n,\mathcal L^{_{(r)}}_{\ell_1,\dots,\ell_r} \big )\le (\ell_r-1)^{\frac 1{\ell_1\cdots \ell_{r-1}}}n^{1-\frac{1}{\ell_1\cdots \ell_{r-1}}}+O\left (n^{\frac 12+\frac 1{2\ell_{r-1}}-\frac 1{\ell_1\cdots \ell_{r-1}}}\right ).$$
\end{thm}

Let us compare Theorem \ref{thm:main-ubound} with  the know upper bounds for the aforementioned cases:

\begin{itemize}
  \item [i)] \emph{$\Ll_{2,2}$-free sets.} The upper bound $|A|\le \sqrt n +O(n^{1/4})$ for any Sidon set $A\subset \{1,\dots,n\}$ was proved by Erd\H os and Tur\'an \cite{E-T} and refined until $|A|< \sqrt n +n^{1/4}+1/2$  by other authors \cite{L,Ru,Ci1}. The Erd\H os-Tur\'an bound 
follows from  Theorem \ref{thm:main-ubound} for $r=\ell_1=\ell_2=2$.

      \smallskip

  \item [ii)] \emph{$\Ll_{2,\ell}$-free sets.} The upper bound $|A|<\sqrt{(\ell-1)n}+((\ell-1)n)^{1/4}+1/2$ for $B_2^{\circ}[\ell-1]$ sets $A\subset \{1,\dots,n\}$ was proved in \cite{Ci1}. Theorem \ref{thm:main-ubound} for $r=\ell_1=2$ and $\ell_2=\ell \ge 2$ gives
      $$F(n,\Ll_{2,\ell})\le (\ell-1)^{1/2}n^{1/2}+O(n^{\frac 14}).$$

      \medskip

    \item [iii)] \emph{$\mathcal L_{\ell_1,\ell_2}$-free sets.} Peng, Tesoro and Timmons \cite{PTT} proved that if $A\subset\{1,\dots,n\}$ does not contain $\ell_1$ copies of any set of $\ell_2$ elements then $|A|\le (\ell_2-1)^{1/\ell_1}n^{1-1/\ell_1}+O\left (n^{1/2-1/(2\ell_1)}\right )$. This also follows from Theorem \ref{thm:main-ubound} for $r=2$. Note that Erd\H os and Harzheim \cite{E-H} had previously proved the weaker estimate $|A|\ll n^{1-1/\ell_1}$ for these sets.

        \medskip

       \item [iv)] \emph{$\mathcal L^{_{(r)}}_{2,\dots,2}$-free sets.} Csaba S\'andor \cite{Sa} proved that if $A\subset \{1,\dots,n\}$ does not contain a Hilbert cube of dimension $r$ then $|A|\le n^{1-1/2^{r-1}}+2n^{1-1/2^{r-2}}$, except for finitely many $n$. Gunderson and R\"odl \cite{GR} had previously established the weaker upper bound  $|A|\ll n^{1-1/2^{r-1}}$.
  Theorem \ref{thm:main-ubound} in the case $\Ll^{(r)}_{2,\dots,2}$ implies $$F(n,\Ll^{(r)}_{2,\dots,2})\le n^{1-1/2^{r-1}}+O\left (n^{3/4-1/2^{r-1}}  \right ),$$ which improves  the error term for $r\ge 4$ in Sandor's estimate.
\end{itemize}

The probabilistic method provides a general lower bound for $F(n,\mathcal L^{_{(r)}}_{\ell_1,\dots,\ell_r} )$.
\begin{thm}\label{thm:main-lbound}For any $r\ge 2$ and $ 2\le \ell_1\le \cdots \le \ell_r$ we have
$$F(n,\mathcal L^{_{(r)}}_{\ell_1,\dots,\ell_r} )\ge n^{1-\frac{\ell_1+\dots+\ell_r\ -r}{\ell_1\cdots \ell_r\ -1}-o(1)}.$$
\end{thm}
The exponents in Theorems \ref{thm:main-ubound} and \ref{thm:main-lbound} are distinct and to close the gap between them is a major problem. We think that the exponent for these extremal sets is the one attained in the upper bound.
\begin{conj}\label{con1}For any $r\ge 2$ and $ 2\le \ell_1\le \cdots \le \ell_r$,
$$F(n,\mathcal L^{_{(r)}}_{\ell_1,\cdots,\ell_r})\asymp n^{1-1/(\ell_1\cdots \ell_{r-1})}.$$
\end{conj}

This conjecture  has been proved for some particular cases:
\begin{align}
F(n,\mathcal L_{2,\ell_2}) &\sim (\ell_2-1)^{1/2}n^{1-1/2}, \label{Ex-2,l2}\\
F(n,\mathcal L_{3,\ell_2}) &\asymp n^{1-1/3}, \quad (\ell_2\ge 3), \label{Ex-3,l2}\\
F(n,\mathcal L_{\ell_1,\ell_2}) &\asymp n^{1-1/\ell_1}, \quad (\ell_2\ge (\ell_1-1)!+1) \label{Ex-l1,l2}.
\end{align}
The asymptotic estimate \eqref{Ex-2,l2} for $\ell_2=2$ recovers the estimate found by Erd\H os and Tur\'an \cite{E-T} for extremal Sidon sets and was generalized to any $\ell_2\ge 2$ by Trujillo-Solarte, Garc\'ia-Pulgar\'in and Vel\'asquez-Soto \cite{GTV}. The estimates \eqref{Ex-3,l2} and \eqref{Ex-l1,l2} were proved by Peng, Tesoro and Timmons \cite{PTT}.

The lower bound in Theorem \ref{thm:main-lbound} has also been improved in other cases although they do not match the exponent $1-1/(\ell_1\dots \ell_{r-1})$. Trivially $F(n,\mathcal L_{4,\ell_2})\ge F(n,\mathcal L_{3,\ell_2})$, thus \eqref{Ex-3,l2} implies $F(n,\mathcal L_{4,\ell_2})\gg n^{1-1/3}$ for $\ell_2\ge 3$, which gives a better lower bound than Theorem \ref{thm:main-lbound} for $\ell_2=4,5,6$.

Another interesting case corresponds to $\Ll^{_{(3)}}_{2,2,2}$-free sets. Theorem \ref{thm:main-lbound} gives the lower bound $F(n,\mathcal L^{_{(3)}}_{2,2,2})\gg n^{1-3/7-o(1)}$ but Katz,  Krop and Maggioni \cite{KKM} found a construction which gives
\begin{equation}\label{lbound:Ex222}
F(n,\mathcal L^{_{(3)}}_{2,2,2})\gg n^{1-1/3}.
\end{equation}
We confirm this last lower bound with an alternative construction based upon a $\Ll^{(3)}_{2,2,2}$-free set in $\Z_{p-1}^3$. 
\begin{defn}Given a finite abelian group $G$, we will denote by $F(G,\mathcal L^{_{(r)}}_{\ell_1,\cdots,\ell_r})$ the largest size of a $\mathcal L^{_{(r)}}_{\ell_1,\cdots,\ell_r}$-free set in $G$.
\end{defn}
\begin{thm}\label{lbound:Ex(Z_{(p-1)}^3,222)}
For any prime $p \ge 2$ we have
$$F(\Z_{p-1}^3,\Ll^{(3)}_{2,2,2})\ge (p-3)^2.$$
\end{thm}
The set we construct to prove Theorem \ref{lbound:Ex(Z_{(p-1)}^3,222)} can be easily projected to the integers to prove \eqref{lbound:Ex222}, as it was done in \cite{KKM}. In general we have
\begin{prop}\label{prop:E(n_1n_2..n_k)}
$$F(2^{k-1}n_1\cdots n_k, \Ll^{(r)}_{\ell_1,\dots,\ell_r})\ge F(\Z_{n_1}\times \cdots \times\Z_{n_k}, \Ll^{(r)}_{\ell_1,\dots,\ell_r}).$$
\end{prop}

\subsection{Extremal problems in graphs and hypergraphs}
Given a graph $\mathcal H$, let $\ex(n,\mathcal H)$ denote the maximum number of edges (or hyperedges) of a $n$ vertices graph (or hypergraph) which does not contain $\mathcal H$ as a sub-graph (or sub-hypergraph). Estimating $\ex(n,\mathcal H)$ is a major problem in extremal graph theory. An important case is when $\mathcal H=K_{\ell_1,\ell_2}$.
It is known  that
\begin{equation}\label{co}n^{2-\frac{\ell_1+\ell_2-2}{\ell_1\ell_2-1}}\ll \ex(n, K_{\ell_1,\ell_2})\le \frac 12(\ell_2-1)^{1/\ell_1}n^{2-\frac 1{\ell_1}} (1+o(1)).\end{equation}
The upper bound was obtained by K\"{o}vari, S\'os and Tur\'an \cite{KST} and the lower bound can be easily obtained using the probabilistic method.

There is a gap between the exponents in \eqref{co} and to improve the exponent on the lower bound is a difficult problem. The conjecture is that the true exponent  is the one attained in the upper bound. The only cases where the upper bound has been reached by a construction of a graph  with $\Omega(n^{2-1/\ell_1})$ edges are
\begin{align}
\ex(n,K_{2,2}) &=\frac 12n^{3/2}(1+o(1)), \label{exK22}\\
\ex(n,K_{2,\ell_2}) &=\frac{\sqrt{\ell_2-1}}2n^{3/2}(1+o(1)),\quad (\ell_2\ge 2), \label{exK(2l2)}\\
\ex(n,K_{3,3}) &=\frac{1}2n^{5/3}(1+o(1)),\label{exK33}\\
\ex(n,K_{\ell_1,\ell_2}) &\asymp n^{2-1/\ell_1},\quad (\ell_2\ge (\ell_1-1)!+1). \label{exK(l1l2)}
\end{align}
Erd\H os, R\'enyi and S\'os \cite{ERS}, and Brown \cite{Br} proved \eqref{exK22}. F\"{u}redi \cite{F} obtained \eqref{exK(2l2)} and Brown \cite{Br} and F\"{u}redi \cite{F} proved \eqref{exK33}, whereas \eqref{exK(l1l2)} was proved by Alon, R\'onyai and Szab\'o \cite{ARS}.
Ball and Pepe \cite{BP} have recently proved that $\ex(n,K_{5,5})\gg n^{7/4}$. Their result also improves the exponent in the lower bound in \eqref{co} for the cases  $(\ell_1,\ell_2)=(5, \ell_2),\ 5\le \ell_2\le 12$ and for $(\ell_1,\ell_2)=(6,\ell_2),\ 6\le \ell_2\le 8$.

The analogous problem for hypergraphs  seems to be more difficult.

\begin{defn} \label{defn:r-uniform_hypergraph}
Let $r\ge 2$ and $2\le \ell_1\le \cdots \le \ell_r$ be integers. We write $K^{(r)}_{\ell_1,\dots,\ell_r}$  for the $r$-uniform hypergraph $(V,\mathcal E)$ where $V=V_1\cup\cdots \cup V_r$ with $|V_i|=\ell_i,\ i=1,\dots,r$ and $$\mathcal E=\left \{\{x_1,\dots,x_r\}:\ x_i\in V_i,\ i=1,\dots,r\right \}.$$

We will say that the $r$-hypergraph $\mathcal H$ is $K^{(r)}_{\ell_1,\dots,\ell_r}$-free when $\mathcal H$ does not contain
any $r$-uniform hypergraph $K^{(r)}_{\ell_1,\dots,\ell_r}$.
\end{defn}
We recall that  $\ex(n; K^{(r)}_{\ell_1,\cdots ,\ell_r})$ is the maximum number of hyperedges of a $K^{(r)}_{\ell_1,\dots,\ell_r}$-free hypergraph of $n$ vertices. See \cite{FS} for a nice survey on extremal problems on graphs and hypergraphs. An easy probabilistic argument gives a lower bound which generalizes \eqref{co}:
\begin{equation}\label{bounds}
n^{r-\frac{\ell_1+\cdots +\ell_r-r}{\ell_1\cdots \ell_r-1}}\ll \ex(n; K^{(r)}_{\ell_1,\cdots ,\ell_r}).
\end{equation}
The upper bound was considered by Erd\H os in the case $\ell=\ell_1=\cdots =\ell_r$. He proved
 \cite[Theorem 1]{E}  that 
\begin{equation}\label{eq:ubound_Kl..l-free_HGraph-E_1}
ex(n,K^{(r)}_{\ell,\ldots,\ell}) \ll n^{r - 1/\ell^{r-1}}.
\end{equation} 
Erd\H os and Simonovits wrote in \cite{ES}  that probably $\lim_{n\to \infty} \frac{\ex(n, K^{(r)}_{\ell,\cdots ,\ell })}{n^{r-1/\ell^{r-1}}}$ exists and it is a positive number.
We refine the estimate \eqref{eq:ubound_Kl..l-free_HGraph-E_1} as follows.
\begin{thm} \label{thm:ubound_Kl..l-free_HGraph}
\begin{equation} \label{ubound_Kl..l-free_HGraph}
ex(n,K^{(r)}_{\ell_1,\ldots,\ell_r}) \le \frac{(\ell_r-1)^{1/\ell_1\cdots \ell_{r-1}}}{r!} \; n^{r - 1/\ell_1\cdots \ell_{r-1}} (1 + o(1)), \quad (n \to \infty). 
\end{equation}
\end{thm}
The case $r=2$ in Theorem \ref{thm:ubound_Kl..l-free_HGraph} is the result \eqref{co} proved by K\"{o}vari, S\'os and Tur\'an \cite{KST}. 
It is believed that the upper bound in \eqref{ubound_Kl..l-free_HGraph} is not far from the real value of $ex(n,K^{(r)}_{\ell_1,\ldots,\ell_r})$.
\begin{conj}\label{con2}
$$\ex(n, K^{(r)}_{\ell_1,\cdots ,\ell_r})\asymp n^{r-1/\ell_1\cdots \ell_{r-1}}.$$
\end{conj}

The lower bound in \eqref{bounds} has been improved in a few cases for $r \ge 3$:
\begin{align}
\ex(n,K_{2,2,2}^{(3)}) &\gg n^{3-1/3}, \label{lbound:exK222}\\
\ex(n,K^{(r)}_{2,\dots,2}) & \gg n^{r-\frac r{2^r-1}+\frac 1{s(2^r-1)}}, \quad (sr\equiv 1\pmod{2^r-1}). \label{lbound:exK2-2}
\end{align}
Katz, Krop and Maggioni \cite{KKM} attained \eqref{lbound:exK222} and Gunderson and R\"{o}dl \cite{GR} proved \eqref{lbound:exK2-2}. An alternative proof of \eqref{lbound:exK222} follows from Theorem \ref{lbound:Ex(Z_{(p-1)}^3,222)} and Proposition \ref{ex-ge-Ex} below.

\subsection{Connection between the two areas of problems}
The two exponents of $n$ in Theorems \ref{thm:main-ubound} and \ref{thm:main-lbound} have the same flavour as the two exponents of $n$ in \eqref{bounds} and in Theorem \ref{thm:ubound_Kl..l-free_HGraph}. This is a consequence of the following result.
\begin{prop}\label{ex-ge-Ex}
Let $G$ be a finite abelian group with $|G|=n$. Then
$$\ex(n,K^{(r)}_{\ell_1,\dots,\ell_r})\ge \binom nr \frac{F(G,\Ll^{(r)}_{\ell_1,\dots,\ell_r})}n.$$
\end{prop}
The proof uses  $\Ll^{(r)}_{\ell_1,\dots,\ell_r}$-free sets in finite abelian groups to construct $K^{(r)}_{\ell_1,\dots,\ell_r}$-free hypergraphs.
Proposition \ref{ex-ge-Ex} connects results on extremal problems in abelian groups with results on extremal problems in hypergraphs.

We mention a couple of cases already discussed in the literature. It is well known that if $A$ is a Sidon set in a finite abelian group $G$ then the graph $\mathcal G(V,\mathcal E)$ where $V=G$ and $\mathcal E=\{\{x,y\}:\ x+y\in A\}$ is a $K_{2,2}$-free graph. Another related example is the connection that was discussed in \cite{PTT} between $\Ll_{\ell_1,\ell_2}$-free sets and the problem of Zarankiewicz, which in turn is connected to extremal problems on $K_{\ell_1,\ell_2}$-free graphs (see \cite[\S 2]{F}). As a final example \eqref{lbound:exK222} can be obtained from Theorem \ref{lbound:Ex(Z_{(p-1)}^3,222)} by using Proposition \ref{ex-ge-Ex}. 

In the same line of reasoning since any $\Ll^{(r)}_{\ell_1,\dots,\ell_r}$-free set in $\Z_n$ is also a $\Ll^{(r)}_{\ell_1,\dots,\ell_r}$-free set in $\{1,\dots, n\}$ then Conjecture \ref{con1} implies Conjecture \ref{con2}. The converse is not true but the algebraic ideas behind the constructions of  some $K_{\ell_1,\ell_2}$-free graphs can be used in some cases to construct $\Ll_{\ell_1,\ell_2}$-free sets. This was the strategy followed in \cite{PTT} to construct dense $\Ll_{3,3}$-free sets.

\subsection{Extremal problems in infinite sequences of integers}
We consider also infinite $\mathcal L$-free sequences of positive integers. This problem is more difficult than the analogous finite problem, even in the simplest case of $\mathcal L_{2,2}$-free sets (Sidon sequences). Let $A(x)=|A \cap [1,x]|$ be the counting function of any sequence $A$. 
In the light of Conjecture \ref{con1} for the finite case and being optimistic one could believe in the existence of an infinite $\mathcal L^{(r)}_{\ell_1,\dots,\ell_r}$-free sequence  satisfying $A(x)\gg x^{1-1/(\ell_1\cdots \ell_{r-1})}$. Erd\H os \cite{St} proved that it is not true for Sidon sequences, and Peng, Tesoro and Timmons \cite{PTT} proved that neither for $\mathcal L_{\ell_1,\ell_2}$-free sequences this is true. We generalize these results for all $\mathcal L^{(r)}_{\ell_1,\dots,\ell_r}$.
\begin{thm} \label{thm:liminf_Lfree_seq}
If $A$ is an infinite $\mathcal L^{(r)}_{\ell_1,\dots,\ell_r}$-free sequence then
$$\liminf_{x\to \infty}\frac{A(x)}x (x\log x)^{1/(\ell_1\cdots \ell_{r-1})} \ll 1.$$
\end{thm}
Hence a natural question is
\emph{whether or not for any $\epsilon>0$ there exists a $\mathcal L^{(r)}_{\ell_1,\dots,\ell_r}$-free sequence with 
\begin{equation} \label{eq:Erdos-conj}
A(x)\gg x^{1-1/(\ell_1\cdots \ell_{r-1})-\epsilon}.
\end{equation}
}
A positive answer to this question was conjectured by Erd\H os in the case of Sidon sequences. The greedy algorithm provides a Sidon sequence $A$ with $A(x)\gg x^{1/3}$. This was the densest infinite Sidon sequences known during nearly 50 years. Ajtai, Koml\'os and Szemer\'edi \cite{AKS} proved the existence of a Sidon sequence with $A(x)\gg (x\log x)^{1/3}$ and Ruzsa \cite{Ru2} proved the existence of a Sidon sequence with $A(x)\gg x^{\sqrt 2-1+o(1)}$. The first author \cite{Ci4} constructed an explicit Sidon sequence with similar counting function. 

To attain the exponent in \eqref{eq:Erdos-conj} looks like a difficult problem. It even seems difficult to get a exponent greater than $1- \frac{\ell_1+\cdots+\ell_r -r}{\ell_1 \cdots \ell_r-1}$, which is the exponent obtained in Theorem \ref{thm:main-lbound} for finite sets.
 
The probabilistic method used in Theorem \ref{thm:main-lbound} to construct dense finite $\mathcal L^{(r)}_{\ell_1,\dots,\ell_r}$-free sets might be 
adapted to construct infinite $\mathcal L^{(r)}_{\ell_1,\dots,\ell_r}$-free sequences with large counting function in order to prove that
for every $\epsilon > 0$  there exist an infinite $\Ll^{(r)}_{\ell_1, \ldots, \ell_r}$-free sequence $A$ satisfying
$$A(x) \gg x^{1-\gamma-\epsilon}, \text{ with } \gamma=\frac{\ell_1+\cdots+\ell_r-r}{\ell_1 \cdots \ell_r-1}.$$

We have not found a proof for the general case, however we have succeeded in two particular cases.
\begin{thm} \label{thm:2L-free_seq}
For any $\ell \ge 2$ and for any $\epsilon >0$ there exists an infinite $\mathcal L_{2,\ell}$-free sequence with $$A(x)\gg x^{1-\frac{\ell}{2\ell-1}-\epsilon}.$$
\end{thm}
Note however that the constructions in \cite{Ru2} and \cite{Ci4} provide a greater exponent for $\ell=2$ and $\ell=3$.

\begin{thm} \label{thm:L2...2-free_seq}
For any $r\ge 2$ and for any $\epsilon >0$ there exists an infinite $\mathcal L^{(r)}_{2,\dots,2}$-free sequence with 
$$A(x)\gg x^{1-\frac{r}{2^{r}-1}-\epsilon}.$$
\end{thm}


%
%
%

\

The rest of this paper is organized as follows.
Several auxiliary results that will be used in the sequel are included in section \ref{sec:lemmata}.
In section \ref{sec:finite-Lfree-sets} we discuss finite sets that are $\Ll$-free and we prove Theorems \ref{thm:main-ubound}, \ref{thm:main-lbound}, \ref{lbound:Ex(Z_{(p-1)}^3,222)}, and Proposition \ref{prop:E(n_1n_2..n_k)}.
In section \ref{sec:Kfree-hypergraphs} we prove Proposition \ref{ex-ge-Ex} and Theorem \ref{thm:ubound_Kl..l-free_HGraph}.
In section \ref{sec:Lfree-sequences} we discuss infinite $\Ll$-free sequences of integers and prove Theorems \ref{thm:liminf_Lfree_seq},  \ref{thm:2L-free_seq} and \ref{thm:L2...2-free_seq}.
\section{Lemmata and auxiliary theorems.} \label{sec:lemmata}
\begin{lem} \label{lem:TBD-1}
Let $r\ge 2$ and $ 2\le \ell_1\le \cdots \le \ell_r$ be given.
If $A$ is a $\Ll_{\ell_1,\ldots,\ell_r}^{(r)}$-free set in an abelian group $G$ then
the set $$(A+x_1)\cap \cdots \cap (A+x_{\ell_1})$$ is a $\mathcal L^{(r-1)}_{\ell_2,\dots,\ell_r}$-free set for any collection $\{x_1,\dots,x_{\ell_1}\}$ of $\ell_1$ distinct elements of $G$.
\end{lem}
\begin{proof}
If $L_2+\cdots +L_{r}\subset (A+x_1)\cap \cdots \cap (A+x_{\ell_1})$ then $L_2+\cdots +L_{r}-x_i\subset A$ for $i=1,\dots ,\ell_1$ and so $L_1+L_2+\cdots +L_{r}\subset  A$ for $L_1=\{-x_1,\dots,-x_{\ell_1}\}$.
\end{proof}

\begin{defn} \label{degenerated-sumset}
We say that a sumset $L_1+\cdots +L_r$ is \emph{degenerated} if $|L_1+\cdots +L_r|<|L_1|\cdots |L_r|$, that is to say some of all the possible sums are repeated.
 \end{defn}
\begin{lem} \label{lem:degen-sumset-contains-AP}
 If a sumset is degenerated then it contains an arithmetic progression.
 \end{lem}
 \begin{proof}Consider the sumset $X=L_1+\cdots +L_r$
Suppose that $x_1+\cdots +x_r=y_1+\cdots +y_r$ with $x_i,y_i\in L_i \; (1 \le i \le r)$, and with $x_k\not=y_k$, say $x_k>y_k$, for some $k$. Then the three following elements of $X$ form an arithmetic progression of difference $x_k-y_k$:
$$x_1+\cdots+x_{k-1}+y_k+x_{k+1}+\cdots+x_r,\quad x_1+\cdots +x_r,\quad y_1+\cdots +y_{k-1}+x_k+y_{k+1}+\cdots+y_r.$$
\end{proof}

 \begin{lem}\label{countsumset}Given $r \ge 2$ and $2\le \ell_1\le \cdots \le \ell_r$, there are at most $n^{\ell_1+\cdots+\ell_r-r+1}$ sumsets $X\subset \{1,\dots ,n\}$ of the form $X=L_1+\cdots +L_r$ with $|L_i|=\ell_i,\ i=1,\dots,r$.
 \end{lem}
 \begin{proof}
 Any sumset $X$ can written in only a way in the form $x+L_1'+\cdots+L_r'$ with $L_i'$ the translate of $L_i$ such that $\min L_i'=0$. The number of choices for $x,L_1',\dots,L_r'$ is at most $n\binom n{\ell_1-1}\cdots \binom n{\ell_r-1}<n^{1+(\ell_1-1)+\cdots +(\ell_r-1)}$.
 \end{proof}

\

We recall a Theorem, which is a consequence of the Jensen's inequality, and that will be used in the proof of Theorem \ref{thm:ubound_Kl..l-free_HGraph}.
\begin{thm}[Overlapping Theorem \cite{C-T}] \label{thm:overlapping-thm}
Let $(\Omega, \mathcal{A}, \mathbb{P})$ be a probability space and let $\{ E_j \}_{j=1}^k$ denote a family of events. Write
\[
\sigma_r := \sum_{1 \le j_1 < \cdots < j_r \le k} \Pr{ E_{j_1} \cap \cdots \cap E_{j_r}}, \quad (r \ge 1).
\]
Then we have
\[
\sigma_r \ge \binom{\sigma_1}{r}=\frac{\sigma_1(\sigma_1-1)\cdots (\sigma_1-(r-1))}{r!}.
\]
\end{thm}

An inmediate corollary of Theorem \ref{thm:overlapping-thm} is the following lemma which will be used several times through this work.
\begin{lem}\label{lem:over}
Let $r \ge 2$ be an integer. If $A+B\subset X$ then
\begin{equation}\label{eq:over}\frac 1{|X|}\sum_{\{x_1,\dots,x_{r}\}\in \binom{B}{r}}|(A+ x_1)\cap \cdots \cap (A+x_{r})|\ge \binom{\frac{|A||B|}{|X|}}{r}.
\end{equation}
\end{lem}

\

Behrend \cite{Behrend} proved the following result that will be used in the random constructions we make in the sequel.
\begin{thm}[Behrend] \label{Behrend} For $n$ large enough, any set  of $n$ consecutive integers contains a subset $B_n$ free of arithmetic progressions with size $|B_n|=n^{1-\omega(n)}$, for some decreasing function $\omega(n)\to 0$ when $n\to \infty.$
\end{thm}
Indeed it is possible to take $\omega(n)\asymp 1/\sqrt{\log n}.$

\section{Finite $\Ll$-free sets} \label{sec:finite-Lfree-sets}
\subsection{Proof of Theorem \ref{thm:main-ubound}}
As the first step in the proof we attain a weaker version of the upper bound of Theorem \ref{thm:main-ubound}.
\begin{lem}\label{lem:weak}
For $r\ge 2$ and $2 \le \ell_1 \le \cdots \le \ell_r$ we have
$$F\big (n,\mathcal L^{(r)}_{\ell_1,\dots,\ell_r}\big )\ll n^{1-1/(\ell_1\cdots \ell_{r-1})}.$$
\end{lem}
\begin{proof}
For short we write $F_r=F\big (n,\mathcal L^{(r)}_{\ell_1,\dots,\ell_r}\big )$ and $\ F_{r-1}=F\big (n,\mathcal L^{(r-1)}_{\ell_2,\dots,\ell_r}\big )$. Suppose that $A\subset [n]$ is a $\mathcal L^{(r)}_{\ell_1,\dots,\ell_r}$-free set of largest cardinality, so we have $|A|=F_r$.
Lemma \ref{lem:TBD-1} implies that
\begin{equation}\label{int}|(A+x_1)\cap \cdots \cap (A+x_{\ell_1})|\le F_{r-1},
\end{equation}
holds for any set of distinct positive integers $\{x_1,\dots, x_{\ell_1}\}$.

Now we take $B=[1,n]$ and $X=[1,2n]$ in Lemma \ref{lem:over} and use \eqref{int} to get
$$\frac 1{2n}\binom n{\ell_1} F_{r-1}\ge \binom{F_r/2}{\ell_1}\implies \frac{n^{\ell_1-1}}2F_{r-1}> (F_r/2-(\ell_1-1))^{\ell_1},$$
and then we have 
\begin{equation}\label{des}F_r< (2n)^{1-1/\ell_1}(F_{r-1})^{1/\ell_1}+2\ell_1.\end{equation}
This inequality allow us to prove Lemma \ref{lem:weak} using induction on $r$. 
First note that $F_1=F(n,\mathcal L_{\ell_2})=\ell_2-1$ for $n\ge \ell_2-1$ and inserting \eqref{des}
we have that the Lemma is true for $r=2$.

Assume that it is true for $r-1$. Inequality \ref{des} implies
$$F_r\ll n^{1-1/\ell_1}(F_{r-1})^{1/\ell_1}\ll n^{1-1/\ell_1}(n^{1-1/(\ell_2\cdots \ell_{r-1})})^{1/\ell_1}\ll n^{1-1/(\ell_1\cdots \ell_{r-1})}.$$
\end{proof}

Next we will prove a refined version of the inequality \eqref{des}.
\begin{lem} \label{lem:strong}
Let $r, \ell_1,\ldots,\ell_r$ be integers with $r\ge 2$ and $2\le \ell_1\le \cdots \le \ell_r$. We have the following inequality:
$$F\big(n,\mathcal L^{(r)}_{\ell_1,\dots,\ell_r}\big)< n^{1-\frac 1{\ell_1}}\big (F\big (n,\mathcal L^{(r-1)}_{\ell_2,\dots,\ell_r}\big)\big )^{\frac 1{\ell_1}}+ O\left ( n^{1/2-1/(2\ell_1\cdots \ell_{r-1})}  \right ) .$$
\end{lem}
\begin{proof}
 The first part of the proof is similar to the proof of Lemma \ref{lem:weak} but now we take  $B=[0,m]$ and $X=[1,n+m]$ in Lemma \ref{lem:over} where $m$  will be choosen later. The inequality \eqref{int} and Lemma \ref{lem:over} imply
\begin{eqnarray*}\frac{F_{r-1}}{n+m}\binom{m+1}{\ell_1}&\ge &\binom{\frac{(m+1)F_r}{n+m}}{\ell_1}\\\implies \frac {F_{r-1}}{n+m}(m+1)^{\ell_1}&\ge &\left (\frac{(m+1)F_r}{n+m}-(\ell_1-1)\right )^{\ell_1} \\ \implies \frac {F_{r-1}}{n+m}&\ge &\left (\frac{F_r}{n+m}-\frac{\ell_1-1}{m+1}\right )^{\ell_1}
\\ \implies  \frac{F_r}{n+m}&\le &\left (\frac{F_{r-1}}{n+m}\right )^{1/\ell_1}+\frac{\ell_1-1}{m+1}.\end{eqnarray*}
Using that $(n+m)^{1-1/\ell_1}=n^{1-1/\ell_1}(1+m/n)^{1-1/\ell_1}<n^{1-1/\ell_1}(1+m/n)$ we get
\begin{align*}
F_r&\le  (n+m)^{1-1/\ell_1}{F_{r-1}}^{1/\ell_1}+\frac{(\ell_1-1)(n+m)}{m+1}\\
&\le n^{1-1/\ell_1}{F_{r-1}}^{1/\ell_1}\left (1+\frac m{n}\right )+\frac{(\ell_1-1)(n+m)}{m+1}\\
&\le n^{1-1/\ell_1}{F_{r-1}}^{1/\ell_1}+ m\left (\frac{F_{r-1}}n\right )^{1/\ell_1}+\frac{(\ell_1-1)n}{m+1} + \ell_1 -1.
\end{align*}

To make as sharp as possible this last estimate we need that the second and the third summands have the same order of magnitude.
Hence by taking
$m=\left \lfloor \sqrt{n\ell_1}/(F_{r-1}/n)^{1/(2\ell_1)}\right \rfloor$
and using also Lemma \ref{lem:weak} we have
\begin{align}
F_r&\le n\left (F_{r-1}/n\right )^{1/\ell_1}+ 2\frac{\ell_1-1}{\sqrt \ell_1}\sqrt n\left (F_{r-1}/n\right )^{1/(2\ell_1)}+\ell_1-1 \label{ineq:1-in-lem:strong} \\
&< n^{1-1/\ell_1}F_{r-1}^{1/\ell_1}+O\big (n^{1/2-1/(2\ell_1\cdots \ell_{r-1})}\big ) \notag,
\end{align}
as claimed.
\end{proof}
To finish the proof of Theorem \ref{thm:main-ubound} we proceed by induction on $r$.
For $r=2$, let $F_2=F(n,\mathcal L_{\ell_1,\ell_2})$  and $F_1=F(n,\mathcal L_{\ell_2})$.
Observe that $F_1=\ell_2-1$ for $n\ge \ell_2-1$. Inequality \eqref{ineq:1-in-lem:strong} implies
$$
F_2 < (\ell_2-1 )^{1/\ell_1}n^{1-1/\ell_1}+ O\left ( n^{1/2-1/(2\ell_1)}  \right ).
$$
Assume that Theorem \ref{thm:main-ubound} is true for $r-1$ and take any $\ell_1,\dots, \ell_r$ with $2\le \ell_1\le\cdots \le \ell_r$. Lemma \ref{lem:strong} and the induction hypothesis imply
\begin{eqnarray*}F_r&<& n^{1-1/\ell_1}\left (F_{r-1}\right )^{1/\ell_1}+  O\left ( n^{1/2-1/(2\ell_1\cdots \ell_{r-1})}  \right ) \\
&<&n^{1-1/\ell_1}\left (   (\ell_r-1)^{  \frac 1{\ell_2\cdots \ell_{r-1}}  }n^{1-\frac{1}{\ell_2\cdots \ell_{r-1}}}+O\left (n^{\frac 12+\frac 1{2\ell_{r-1}}-\frac 1{\ell_{2}\cdots \ell_{r-1}}}\right )      \right )^{1/\ell_1}\\&+& O\left (n^{\frac 12-\frac 1{2\ell_1\cdots \ell_{r-1}}}  \right )\\
&<&(\ell_r-1)^{\frac 1{\ell_1\cdots \ell_{r-1}}}
n^{1-\frac 1{\ell_1\cdots \ell_{r-1}}}
\left ( 1+ O\left (    n^{   -\frac 12+\frac 1{2\ell_{r-1}}   }  \right )    \right )^{1/\ell_1}+
O\left (n^{\frac 12-\frac 1{2\ell_1\cdots \ell_{r-1}}}\right )\\
&<&(\ell_r-1)^{\frac 1{\ell_1\cdots \ell_{r-1}}}
n^{1-\frac 1{\ell_1\cdots \ell_{r-1}}}+O\left (n^{\frac 12+\frac 1{2\ell_{r-1}}-\frac 1{\ell_1\cdots \ell_{r-1}}}    \right )+O\left (n^{\frac 12-\frac 1{2\ell_1\cdots \ell_{r-1}}}\right )
 \end{eqnarray*}
For $r=2$ the second and third summands are the same. For $r>2$ note that 
$
 \ell_1\cdots\ell_{r-2} -2 > -1,
$ and dividing this inequality by $2\ell_1\cdots\ell_{r-1}$ we have that the exponent in the third summand is smaller than the exponent in the second summand. Hence we have
$$
F_r < (\ell_r-1)^{\frac 1{\ell_1\cdots \ell_{r-1}}}
n^{1-\frac 1{\ell_1\cdots \ell_{r-1}}}+O\left (n^{\frac 12+\frac 1{2\ell_{r-1}}-\frac 1{\ell_1\cdots \ell_{r-1}}}    \right ),
$$
as claimed.
 \subsection{Proof of Theorem \ref{thm:main-lbound}}
The lower bound of Theorem \ref{thm:main-lbound} comes from a probabilistic construction. Our proof is a generalization of the argument that Gunderson and Rodl \cite{GR} used for the particular case $\Ll^{(r)}_{2,\dots,2}$.

Let $B_n$ denote the set given by Theorem \ref{Behrend} of Berend, with the following properties: $B_n\subset [n]$, $B$ is free of arithmetic progressions and its size is  $|B_n|=n^{1-\omega(n)}.$ with $\omega(n)=o(1).$ 

Next we randomly construct a set $S$ in $[n]$ with the following probability law:
$$\mathbb P(\nu\in S)=\begin{cases}p \text{ if } \nu\in B_n,\\ 0 \text{ otherwise},\end{cases}$$
where all the events $\{\nu\in S\}_{\nu \ge 1}$ are mutually independent. For the value of $p$ we choose
$$
p=p(n)=\frac 12n^{\frac{r-(\ell_1+\cdots+\ell_r)-\omega(n)}{\ell_1\cdots \ell_r-1}}.
$$
We will say that $X$ is an \emph{obstruction} (for $S$) when $X \subset S$ is a sumset of the class $\Ll_{\ell_1,\ldots,\ell_r} ^{(r)}.$ Our aim is to destroy all the obstructions for $S$. To this end we will remove from $S$ the greatest element  of every obstruction. Let $S^{\mathrm{bad}}$  denote the collection of all these greatest elements:  
\begin{equation} \label{S_b}
S^{\mathrm{bad}} := \bigcup_{ \substack{X \subset S \\ X \in \Ll_{\ell_1,\ldots,\ell_r}^{(r)}}} \{ \max(X)\}.
\end{equation}
If the number of obstructions is low enough, then we could remove from $S$ all the elements in the collection $S^{\mathrm{bad}}$ and still retain a sufficiently dense set which would be free of obstructions. With this in mind we claim that the obstructions for $S$ are few in our construction.
\begin{lem} \label{lem:atleast1/4}
For all $n$ sufficiently large, with probability at least $1/4$ the random sets constructed in this way satisfy
$$ |S| \ge \frac{|B_n|p}{2} \quad \text{ and }  \quad |S^{\mathrm{bad}}| \le \frac{|B_n|p}{4}.$$
\end{lem}
The lemma implies that there exist a set $S \subset [n]$ such that
$$
|S \setminus S^{\mathrm{bad}}| > \frac{|B_n|p}{4} =n^{1-\frac{\ell_1+\cdots+\ell_r-r}{\ell_1\cdots \ell_r-1}-o(1)}.
$$
Note that the set $A=S \setminus S^{\mathrm{bad}}$ satisfies the conditions of Theorem \ref{thm:main-lbound}. Indeed by removing from $S$ all the elements in $S^{\mathrm{bad}}$ we have that $A$ is $\Ll_{\ell_1,\ldots,\ell_r}^{(r)}$-free because all the sumsets of the class $\Ll_{\ell_1,\ldots,\ell_r}^{(r)}$ that were contained in $S$ have been destroyed.

Thus to complete the proof of Theorem \ref{thm:main-lbound} all we need is to prove
 Lemma \ref{lem:atleast1/4}.
\begin{proof}[Proof of Lemma \ref{lem:atleast1/4}]
On the one hand since $S\subset B$, then $S$ is free of arithmetic progressions and Lemma \ref{lem:degen-sumset-contains-AP} implies that all the obstructions for $S$ are non degenerated (see definition \ref{degenerated-sumset}). Hence all the possible sums within an obstruction $X$ are distinct and so the probability of any obstruction $X$ to occur in the construction is
\begin{equation}\label{pp}
\mathbb P(X\subset S \colon X \in \Ll_{\ell_1,\ldots\ell_r}^{(r)})=p^{\ell_1\cdots \ell_r}.
\end{equation}
Consider the random variable $N(S)=\#\{X\subset S \colon X \in \Ll_{\ell_1,\ldots\ell_r}^{(r)}\}$ that counts the number of obstructions. As two different obstructions may have the same maximum then $N(S)$ is greater or equal than the cardinality of $S^{\mathrm{bad}}$.
Hence using Lemma \ref{countsumset} and \eqref{pp} we can estimate the expected cardinal of $S^{\mathrm{bad}}$ as follows: 
\begin{eqnarray*}
\Esp{|S^{\mathrm{bad}}|} &\le & \mathbb E(N(S))=\#\{X\subset [n]:\ X \in \Ll_{\ell_1,\ldots\ell_r}^{(r)}\} \; \mathbb P(X\subset S)\\ &\le & n^{\ell_1+\cdots +\ell_r-r+1}p^{\ell_1\cdots \ell_r}\\
&=& 2^{-\ell_1 \cdots \ell_r} n^{\ell_1+\cdots +\ell_r-r+1}n^{\frac{r-(\ell_1+\cdots+\ell_r)-\omega(n)}{\ell_1\cdots \ell_r-1}(\ell_1 \cdots \ell_r-1+1)}\\
&=&2^{-\ell_1 \cdots \ell_r}n^{1-\omega(n)} n^{\frac{r-(\ell_1+\cdots+\ell_r)-\omega(n)}{\ell_1\cdots \ell_r-1}} =|B_n|p/2^{\ell_1 \cdots \ell_r-1}.
\end{eqnarray*}
Using  the fact that $2 \le \ell_i \; (2\le i \le r)$, this last estimate of $\E(|S^{\mathrm{bad}}|)$, and Markov inequality we have
\begin{align}\label{Markov}
\Pr{|S^{\mathrm{bad}}|>\frac{|B_n|p}4}	& \le \Pr{|S^{\mathrm{bad}}|>\frac{|B_n|p}{2^{\ell_1 \cdots \ell_r-2}}}\\
										&\le \Pr{|S^{\mathrm{bad}}|>2 \; \E(|S^{\mathrm{bad}}|)} \le 1/2.\nonumber
\end{align}
On the other hand the size of $S$ equals $\sum_{\nu \in B_n} 1_{\nu \in S}$, i.e: the sum of $|B_n|$ independent random indicator variables all having the same expectation $p$ and variance $p(1-p)$. This implies $\mathbb E(|S|)= |B_n|p$ and  $\mathrm{Var}(|S|)=|B_n|p(1-p).$ We can now use Chebychev inequality to write
\begin{align}\label{Che}
\Pr{|S|<\frac{|B_n|p}2 }&=\Pr{|S|<\frac{\E(|S|)}2} <\Pr{(|S-\E(|S|)|>\frac{\E(|S|)}2} \notag\\
&<\frac{4\text{Var}(|S|)}{(\E(|S|))^2}=\frac{4|B_n|p(1-p)}{(|B_n|p)^2}<\frac{4}{|B_n|p}<\frac 14,
\end{align} except for finitely many $n$, since $ |B_n|p \to \infty.$
Combining \eqref{Markov} and \eqref{Che} we obtain
$$ \Pr{|S|\ge |B_n|p/2 \text{ and } |S^{\mathrm{bad}}|\le |B_n|p/4} \ge 1-\left (1/2+1/4\right )\ge 1/4, $$
as claimed.
\end{proof}
\subsection{Proof of Theorem \ref{lbound:Ex(Z_{(p-1)}^3,222)}}
Theorem \ref{lbound:Ex(Z_{(p-1)}^3,222)} is an immediate corollary of the following result.
\begin{prop}\label{H3}Let $p$ be an odd prime and $\theta$ be a generator of $\F_p^*$. The set $$A=\{(x_1,x_2,x_3):\ \theta^{x_1}+\theta^{x_2}+\theta^{x_3}=1,\ x_1,x_2,x_3\ne 0\}\subset \Z_{p-1}^3$$ does not contain subsets of the form $L_1+L_2+L_3$ with $|L_1|=|L_2|=|L_3|=2$ and has $(p-3)^2$ elements.
\end{prop}\begin{proof}The first observation is that, given any abelian group $G$, a set $A\subset G$ is free of subsets of the form $L_1+L_2+L_3,\ |L_1|=|L_2|=|L_3|=2$ if and only if
for any $y\in G,\ y\ne 0$, the sets $A^{y}=A\cap (A+y)$ do not contain subsets of the form $L_1+L_2,\ |L_1|=|L_2|=2$. This last condition is equivalent to say that $A^y$ is a Sidon set. This is what we have to prove.

For a fixed $y=(y_1,y_2,y_3)\ne (0,0,0)$ in $\Z_{p-1}^3$ we consider the set $A^y=A\cap (A+y)$.
It is clear from the definition that
$$A^y=\{(x_1,x_2,x_3): \text{ satisfying the conditions } (*)\}$$
$$(*)\qquad \begin{cases}\theta^{x_1}+\theta^{x_2}+\theta^{x_3}=1\\
 \theta^{x_1+y_1}+\theta^{x_2+y_2}+\theta^{x_3+y_3}=1\\
 x_i, x_i+y_i\ne 0,\ i=1,2,3. \end{cases}$$
We claim that, if $A^y$ is not empty, then one of the coordinates of $y$ is distinct from $0$ and distinct from the other two coordinates.

To prove this claim suppose that $(x_1,x_2,x_3)\in A^y$. Since $y=(y_1,y_2,y_3) \not = (0,0,0)$ we can assume that one of the coordinates is different from zero, say $y_3 \ne 0$. If $y_3\ne y_1$ and $y_3\ne y_2$ then the coordinate $y_3$ satisfies the statement of the claim. We consider now the possibility that $y_3=y_2$ (the case $y_3=y_1$ is similar). We will see that in this case, the coordinate $y_1$ satisfies the claim. If $y_1=0$ then the equations (*) imply that $\theta^{x_1}+\theta^{x_2}=\theta^{y_3}(\theta^{x_1}+\theta^{x_2})$ and then, since $y_3\ne 0$, we  have that $\theta^{x_1}+\theta^{x_2}=0$. But it implies that $x_3=0$ which is not possible by construction.  Furthermore it is clear that $y_1\ne y_2= y_3$. Otherwise we would have that $y_1=y_2=y_3$ and the equations (*) would imply that $y=(0,0,0$) which is a contradiction.

Let us assume that $y_3$ is distinct from $0$ and distinct from the other two coordinates. This implies that the elements $\lambda_1=\theta^{y_3}-\theta^{y_1}$, $\lambda_2=\theta^{y_3}-\theta^{y_2}$ and $\mu=\theta^{y_3}-1$ are distinct from zero. Hence taking the function $x_3(x_1,x_2)=\log_{\theta}(1-\theta^{x_1}-\theta^{x_2})$ we can deduce from the equations in (*) that the set $A^y$ is included in the set
$$ S=\{(x_1,x_2,x_3(x_1,x_2)):\ \lambda_1\theta^{x_1}+\lambda_2\theta^{x_2}=\mu,\ \theta^{x_1}+\theta^{x_2}\ne 1\}.$$

Next we show  that $S$ is a Sidon set, which implies that $A^y$ is a Sidon set.
For a given $(z_1,z_2,z_3)\ne (0,0,0)$, suppose that
\begin{equation}\label{si}(x_1,x_2,x_3(x_1,x_2))-(x_1',x_2',x_3(x_1',x_2'))=(z_1,z_2,z_3)\end{equation} with
\begin{equation}\label{sis}\begin{cases}\lambda_1\theta^{x_1}+\lambda_2\theta^{x_2}=\mu,\\ \lambda_1\theta^{x_1'}+\lambda_2\theta^{x_2'}=\mu.
\end{cases}\end{equation}
We will show that $(x_1,x_2,x_3(x_1,x_2))$ and $(x_1',x_2',x_3(x_1',x_2'))$ are uniquely determined by the conditions \eqref{si} and \eqref{sis}.
If $z_1=z_2=0$ then $(x_1,x_2)=(x_1',x_2')$, which implies that $z_3=0$.
 Therefore we can assume that $(z_1,z_2)\ne (0,0)$. In this case  equations in \eqref{si} and \eqref{sis} imply that $$\lambda_2\theta^{x_2}(1-\theta^{z_1-z_2})=\mu(1-\theta^{z_1}).$$ If $z_1=z_2$ then $\mu(1-\theta^{z_1})=0\implies z_1=z_2=0$. If $z_1\ne z_2$ then $x_2$ is uniquely determined and therefore also $x_1,x_1',x_2',x_3(x_1,x_2)$ and $x_3(x_1',x_2')$.

To complete the proof of the lemma we calculate the size of $A$:
\begin{eqnarray*}|A|&=&|\{(u,v,w)\in \F_p^3: u+v+w=1,\ u,v,w \not \in \{0,1\}\}|\\ &=&\sum_{w\not \in \{0,1\}}\sum_{v\not \in \{0,1,-w,1-w\}}1=\sum_{w\not \in \{0,1\}}\left (p-|\{ 0,1,-w,1-w\}|\right )\\ &=&p(p-2)-\sum_{w\not \in \{0,1,-1\}}|\{ 0,1,-w,1-w\}|-|\{ 0,1,1,2\}|\\
&=&p(p-2)-4(p-3)-3=(p-3)^2 .\end{eqnarray*}
\end{proof}

In the rest of this section we will prove the estimate \eqref{lbound:Ex222} by using Theorem \ref{lbound:Ex(Z_{(p-1)}^3,222)} and the following projection from $\Z_m^3$ to $\Z$.
For any integer $m\ge 2$ we consider the function
$\varphi_m:\Z_m^3\to \Z$ defined by
\begin{equation} \label{map:varphi}
\varphi_m(x_1,x_2,x_3)=(2m)^2x_1+(2m)x_2+x_3,
\end{equation}
where $x_1,x_2,x_3$ are  residues in $[0,m-1]$.
An easy, but important, property of this function is that
\begin{equation}\label{o}
\varphi_m(x)+\varphi_m(y)=\varphi_m(u)+\varphi_m(v)\implies x+y=u+v.
\end{equation}
This property  in a more general fashion is proved during the proof of Proposition 1 in section \ref{sec:proof-prop:E(n_1n_2..n_k)}.
\begin{lem}\label{T} If $A\subset \Z_m^3 $ is $\Ll_{2,2,2}^{(3)}$-free
 then the set $\varphi_m(A)$ is $\Ll_{2,2,2}^{(3)}$-free over the integers.
\end{lem}
\begin{proof}
A Hilbert cube of dimension $3$ can be also described as a multiset $\{x_1,\dots,x_8\}$ with $x_2,x_3,x_5\ne x_1$ satisfying the following conditions (see the picture below):
$$\begin{cases}
x_2-x_1=x_4-x_3=x_6-x_5=x_8-x_7\\
x_3-x_1=x_7-x_5
\end{cases}$$

\medskip

\begin{picture}(0,0)\put(80,0){\circle*{5}}\put(95,0){\circle*{5}}\put(120,0){\circle*{5}}\put(135,0){\circle*{5}}
\put(230,0){\circle*{5}}\put(245,0){\circle*{5}}\put(270,0){\circle*{5}}\put(285,0){\circle*{5}}
\put(77,-15){$x_1$}\put(92,-15){$x_2$}\put(117,-15){$x_3$}\put(132,-15){$x_4$}
\put(227,-15){$x_5$}\put(242,-15){$x_6$}\put(267,-15){$x_7$}\put(285,-15){$x_8$}
\end{picture}

\medskip

\

This system of equations is equivalent to the following system in term of sums:
\begin{equation}\label{sumas}\begin{cases}
x_2+x_3=x_4+x_1,\\
x_2+x_5=x_6+x_1,\\
x_2+x_7=x_8+x_1,\\
x_3+x_5=x_7+x_1.
\end{cases}\end{equation}

Suppose that $\{\varphi_m(a_1),\dots, \varphi_m(a_8)\}$  with $a_1,\dots, a_8\in A$ is a Hilbert cube of dimension $3$ contained in $\varphi_m(A)$. Since the elements $\varphi_m(a_1),\dots,$ $\varphi_m(a_8)$ satisfy the  four equations in \eqref{sumas}, the observation  \eqref{o} implies that also the elements  $a_1,\dots, a_8$ satisfy the analogous equations in $\Z_m^3$ and therefore the multiset $\{a_1,\dots, a_8\}$   is a Hilbert cube of dimension $3$  contained in $A$.
\end{proof}

To attain bound \eqref{lbound:Ex222} we apply Lemma \ref{T} to the set $A$ described in Proposition \ref{H3}. It is easy to see that $\varphi_m$ defined by \eqref{map:varphi} is injective thus $|\varphi_{p-1}(A)|=|A|=(p-3)^2$ and we also have
$\varphi_{p-1}(A)\subset [0,4(p-1)^3)$.
Let $p$ be the largest prime $p$ such that $4(p-1)^3\le n$, that is $p=(n/4)^{1/3}(1+o(1))$. Then we have $$F(n,\Ll_{2,2,2}^{(3)})\ge |\varphi_{p-1}(A)|=(p-3)^2= (n/4)^{2/3}(1+o(1)).$$

\subsection{Proof of Proposition \ref{prop:E(n_1n_2..n_k)}} \label{sec:proof-prop:E(n_1n_2..n_k)}
To prepare for the proof we first generalize the idea that we used during the proof of Lemma \ref{T}. The sums in a sumset might be repeated thus we have a multiset in the general case. However a minimum number of the sums are distinct. The set $L_1$ is translated $\ell_2$ times to form a pattern which is in turn translated $\ell_3$ times and so on. The next lemma characterizes sumsets as multisets satisfying a number of conditions.

\begin{lem} \label{lem:sumset-as-multiset} Let $r\ge 2$ and $ 2\le \ell_1\le \cdots \le \ell_r$ be given, and 
a  sumset $L_1+\cdots+L_r \in \Ll^{(r)}_{\ell_1,\ldots,\ell_r}$, with summands $L_s= \{ \lambda_{s1},\ldots, \lambda_{s\ell_s} \}, \; (1 \le s \le r).$ 
We enumerate the sums using a multi-index as follows:
$$x_{i_1 i_2 \cdots i_r} := \sum_{s=1}^r \lambda_{s i_s}, \quad (1 \le i_s \le \ell_s, \: 1 \le s \le r).$$
Then the following $\sum_{k=1}^r \binom{\ell_k}{2}$ conditions hold:
\begin{align}
x_{1 \cdots 1 i_s 1 \cdots 1} &\not = x_{1 \cdots 1 j_s 1 \cdots 1}, \quad (1\le i_s < j_s \le \ell_s, s=1, \ldots, r). \label{eq:ss-ms-g_distinct}
\end{align}  
Furthermore the following $\ell_1 \cdots \ell_r -(\ell_1+\cdots + \ell_r) + (r-1)$ equalities hold:
\begin{align}
x_{1\cdots 1}+ x_{1 \cdots 1 i_s  i_{s+1}\cdots i_r}&= x_{1 \cdots 1 i_s 1  \cdots 1} + x_{1 \cdots 1 i_{s+1} \cdots i_r},  \quad \big(i_s \ne 1, \; (i_{s+1}, \ldots ,i_r) \not = (1,\dots, 1), 1 \le s <r \big). \label{eq:ss-ms-g_sums}
\end{align} 
In the other direction, suppose that a given multiset $X$ of $\ell_1 \cdots \ell_r$ elements can be multi-indexed as follows  
$$X = \lbrace x_{i_1 \cdots i_r},  \quad  1 \le i_s \le \ell_s, 1 \le s \le r \rbrace,$$  so that conditions  \eqref{eq:ss-ms-g_distinct} and \eqref{eq:ss-ms-g_sums} are satisfied. Then $ X \in \Ll^{(r)}_{\ell_1, \ldots,\ell_r}$. 
\end{lem}
\begin{proof}
For $s=1, \ldots, r$  the cardinality of any translate of $L_s$ is $\ell_s$. Hence all the elements in the translate    
\begin{align*}
\sum_{ \substack { 1 \le t \le r \\ t \not = s} } \lambda_{t1} +  L_s  = \lbrace x_{1 \cdots 1 i_s 1 \cdots 1} \colon i_s=1, \ldots, \ell_s  \rbrace
\end{align*}
must be distinct and so \eqref{eq:ss-ms-g_distinct} holds.

Let us fix $s$, with $s \le r-1$, and also fix $i_{s+1},\ldots ,i_r$, with $ (i_{s+1},\ldots ,i_r) \not = (1,\dots, 1)$. Varying just the $s^{\text{th}}$ place of the multi-index we have that the following $\ell_s-1$ equalities on differences hold:
$$x_{1 \cdots 1 i_s i_{s+1}  \cdots i_r}-x_{1 \cdots 1 1 i_{s+1}  \cdots i_r}=\lambda_{si_s}-\lambda_{s1}=x_{1 \cdots 1 i_s 1  \cdots 1}-x_{1 \cdots 1}, \quad ( i_s \not = 1),$$
We now vary $i_{s+1},\ldots ,i_r$,  and we have that  the following $(\ell_{s+1}\cdots \ell_r -1)(\ell_s-1)$ equalities on sums hold:
$$x_{1\cdots 1}+ x_{1 \cdots 1 i_s i _{s+1}\cdots i_r}= x_{1 \cdots 1 i_s 1  \cdots 1} + x_{1 \cdots 1 i_{s+1} \cdots i_r},\quad \big(i_s \ne 1, \: (i_{s+1}, \ldots ,i_r) \not = (1,\dots, 1)\big).$$
As this is true for  $1 \le s < r$ we have \eqref{eq:ss-ms-g_sums}.
Summing in $s$, the total number of equalities is
\begin{align*}
\sum_{s=1}^{r-1} &(\ell_{s+1} \cdots \ell_r -1)(\ell_s-1) = \sum_{s=1}^{r-1} \ell_s \cdots \ell_r -\ell_{s+1} \cdots \ell_r -\ell_s+1  \\
&= \ell_1 \cdots \ell_r  + \sum_{s=2}^{r-1}   - \ell_{s} \cdots \ell_r + \ell_{s} \cdots \ell_r  + \sum_{s=1}^{r-1} (-\ell_s + 1)\\
&=\ell_1 \cdots \ell_r -(\ell_1+\cdots + \ell_r) + (r-1),
\end{align*}
as claimed.

To prove the second part of the lemma, suppose $X = \lbrace x_{i_1 \cdots i_r},  \;  1 \le i_s \le \ell_s, 1 \le s \le r \rbrace$  satisfies conditions  \eqref{eq:ss-ms-g_distinct} and \eqref{eq:ss-ms-g_sums}. We define $L_1 :=  \lbrace x_{11 \cdots 1},x_{21 \cdots 1}, \ldots, x_{\ell_1 1 \cdots 1}\rbrace$ and for $s=2,\ldots,r$ we define
$$
L_s :=  - x_{11 \cdots 1} + \lbrace x_{1 \cdots 1 i_s 1  \cdots 1} \colon i_s=1,\ldots,\ell_s \rbrace.
$$
Condition \eqref{eq:ss-ms-g_distinct} implies that $|L_s| = \ell_s \; (1 \le s \le r),$ and so $L_1+\cdots+L_r \in \Ll^{(r)}_{\ell_1, \ldots\ell_r}$. It will suffice to show that $x_{i_1 \cdots i_r} \in L_1+\cdots+L_r$ for any  element of $X$. 
Note that $x_{1 \cdots 1} \in L_1$ and $0 \in \cap_{s=2}^r L_s$ and thus trivially $x_{1 \cdots 1} \in L_1+\cdots+L_r.$ 
In other case $(i_1, \ldots, i_r) \not = (1, \ldots, 1)$ and then using $r-1$ times \eqref{eq:ss-ms-g_sums} we can write
\begin{align*}
x_{i_1 \cdots i_r} &= x_{i_1 1  \cdots 1} -x_{1\cdots 1}+ x_{1 i_2 \cdots i_r}\\
&= x_{i_1 1  \cdots 1} -x_{1\cdots 1}+ x_{1 i_2 1 \cdots 1} -x_{1\cdots 1} +  x_{1 1 i_3 \cdots i_r}\\
&\vdots\\
&= x_{i_1 1  \cdots 1} +( -x_{1\cdots 1}+ x_{1 i_2 1 \cdots 1}) +\cdots +( - x_{1\cdots 1} + x_{1 \cdots 1 i_r})\\
& \in L_1+L_2+\cdots+L_r. 
\end{align*} \qedhere
\end{proof}

Now we are ready to prove Proposition \ref{prop:E(n_1n_2..n_k)}.
We remind that \emph{given $r_1,\cdots,r_j,\cdots$ (the base), any non negative integer can be written in a unique way in the form
\[
y_1+y_2r_1+y_3r_1r_2+\cdots+y_jr_1r_2\cdots r_{j-1}+\cdots,
\]
with digits $y_j$ satisfying $0\le y_j < r_j \; (j\ge1)$}.

The map $ \varphi \colon \Z_{n_1}\times \cdots \times\Z_{n_k} \to \Z$ given by
$$
\varphi(x_1,\ldots,x_k)=x_1 + x_2 (2 n_1)  +\cdots+ x_k(2 n_1)(2 n_2)\cdots (2 n_{k-1}),
$$
is injective, to see it just note $x_1,\ldots,x_k$ are the $k$ digits of the integer $\varphi(x_1,\ldots,x_k)$ in any base starting with $\lbrace 2n_1, 2n_2, \ldots,2n_{k-1},2n_k \rbrace$.
One important property of $\varphi$ is
\begin{equation} \label{eq:1}
\varphi(x) + \varphi(y)=\varphi(u) + \varphi(v) \Rightarrow x+y=u+v.
\end{equation}
To prove it suppose $\varphi(x) + \varphi(y)=\varphi(u) + \varphi(v)$, that is to say
\begin{equation} \label{eq:2}
\sum_{1}^{k}  (x_j+y_j)(2 n_1)(2 n_2)\cdots (2 n_{j-1})=\sum_{1}^{k}  (u_j+v_j)(2 n_1)(2 n_2)\cdots (2 n_{j-1}).
\end{equation}
For every $j=1,\ldots, k$ we have $0 \le x_j, y_j, u_j, v_j < n_j$ which  implies $$0 \le x_j+y_j, u_j+v_j < 2n_j.$$ The expresion of an integer in the base is unique, thus by \eqref{eq:2} we have $x_i+y_j = u_j+v_j \: (1\le j \le k),$ which implies that $x+y=u+v.$

We claim that $\varphi$ preserves the property of being $\Ll_{\ell_1,\ldots, \ell_r}^{(r)}$-free.
Indeed, let $A \subset \Z_{n_1}\times \cdots \times\Z_{n_k}$ be any set such that $\varphi(A)$ contains a sumset $Y$ of the class $\Ll_{\ell_1,\ldots, \ell_r}^{(r)}$. By Lemma \ref{lem:sumset-as-multiset}, $Y$ can be multi-indexed $\lbrace y_{i_1 \cdots i_r}, \colon 1 \le i_s \le \ell_s, 1 \le s \le r \rbrace$  and satisfies conditions \eqref{eq:ss-ms-g_distinct} and \eqref{eq:ss-ms-g_sums}. The function $\varphi$ is injective and so $\varphi^{-1}(Y)=Z$ can multi-indexed in the natural way : $z_{i_1 \cdots i_r} = \varphi^{-1}(y_{i_1 \cdots i_r}).$ Because $\varphi$ is injective we have that $Z$ satisfies the conditions \eqref{eq:ss-ms-g_distinct}. The property  \eqref{eq:1} implies that $Z$ satisfies the equalities \eqref{eq:ss-ms-g_sums}. Hence the second part of Lemma \ref{lem:sumset-as-multiset} implies that $Z=\varphi^{-1}(Y)$ is  a sumset of the class $\Ll_{\ell_1,\ldots, \ell_r}^{(r)}$ in the group $\Z_{n_1}\times \cdots \times\Z_{n_k}.$

Note that the image of $\varphi$ is in the interval $[2^{k-1} n_1 n_2 \cdots n_k],$ because we have $x_i \le n_i-1, (1\le i \le k)$ and then
\begin{align*}
\varphi(x_1,\ldots, x_k) &\le (n_1-1)+(n_2-1)(2n_1)+(n_3-1)(2n_1)(2n_2)+\cdots\\
&= -1 +n_1 -2n_1 + 2n_1n_2 - 4n_1n_2 + 4n_1 n_2 n_3 +\cdots\\
&=-1 -n_1- 2n_1n_2 - \cdots - 2^{k-2}n_1n_2\cdots n_{k-1} + 2^{k-1}n_1 n_2 \cdots n_k\\
&< 2^{k-1}n_1 n_2 \cdots n_k.
\end{align*}

Therefore for any $\Ll_{\ell_1,\ldots, \ell_r}^{(r)}$-free set $A \subset \Z_{n_1}\times \cdots \times\Z_{n_k}$ we have
\[ |A|=|\varphi(A)| \le F(2^{k-1} n_1 n_2 \cdots n_k, \Ll^{(r)}_{\ell_1,\dots,\ell_r}),
\]
which implies $ F(\Z_{n_1}\times \cdots \times\Z_{n_k}, \Ll_{\ell_1,\ldots, \ell_r}^{(r)})\le F(2^{k-1} n_1 n_2 \cdots n_k, \Ll^{(r)}_{\ell_1,\dots,\ell_r})$.

\section{Connections with Tur\'an problem on graphs and hypergraphs} \label{sec:Kfree-hypergraphs}
\subsection{Proof of Proposition \ref{ex-ge-Ex}}
Let $A\subset G$ be any set free of subsets of the form $L_1+\cdots +L_r$, with $|L_i|=\ell_i,\ i=1,\dots,r.$
Consider the hyper-graph $\mathcal G=(V,\mathcal E)$ where $V=G$ and $$\mathcal E=\left \{\{x_1,\dots,x_r\}\in \binom{G}r:\ x_1+\cdots+x_r\in A\right \}.$$ We claim that $\mathcal G$ does not contain a copy of the $r$-uniform hyper-graph $K^{(r)}_{\ell_1,\dots,\ell_r}$ (see definition \ref{defn:r-uniform_hypergraph}). Otherwise there exist $L_1,\dots, L_r$ with $|L_i|=\ell_i,\ i=1,\dots,r,$ such that all the hyper-edges $\{x_1,\dots,x_r\}$ with $x_i\in L_i$ belong to $\mathcal E$. But this is equivalent to say that $x_1+\cdots +x_r\in A$ for all $(x_1,\dots, x_r)\in L_1\times \cdots \times L_r$. In other words, that $L_1+\cdots+L_r\subset A$, which is not possible because $A$ is $\Ll_{\ell_1,\ldots,\ell_r}^{(r)}$-free. Hence we have
$$ \ex\left (n;K^{(r)}_{\ell_1,\dots,\ell_r}\right )\ge \#\left \{ \{x_1,\dots, x_r\}\in \binom{G}r:\ x_1+\cdots +x_r\in A\right \},$$
an inequality which can be alternatively written as follows
$$ \ex(n;K^{(r)}_{\ell_1,\dots,\ell_r})\ge \sum_{y\in A}R_r(y),$$ where $R_r(y)=\#\{\{x_1,\dots,x_r\}\in \binom{G}r:\ x_1+\cdots +x_r=y\}$.

Note that, for any given $x\in G$, as $A$ is $\Ll_{\ell_1,\ldots,\ell_r}^{(r)}$-free then $A+x$ has the same property. This implies the last inequality  also holds if we sum in $y \in A+x$.  Hence we can write
\begin{eqnarray*}
\ex(n;K^{(r)}_{\ell_1,\dots,\ell_r})&\ge &\frac 1{|G|}\sum_{x\in G}\sum_{y\in A+x}R_r(y)=\frac 1{|G|} \sum_{y\in G}R_r(y)\#\{x:\ x\in y-A\}
\\ &=&\frac{|A|}{|G|} \sum_{y\in G}R_r(y)=\frac{|A|}{|G|} \binom{|G|}r=\frac{|A|}n\binom nr.
\end{eqnarray*}

\subsection{Proof of Theorem \ref{thm:ubound_Kl..l-free_HGraph}}

The proof uses induction in $r$. The case $r=2$  was proved by K\"ovari, S\'os and Tur\'an \cite{KST}.

For $r\ge 3$ to ease the notation we write $e_{r}=ex(n,K^{(r)}_{\ell_1,\ldots,\ell_r})$ and $e_{r-1}=ex(n,K^{(r-1)}_{\ell_2,\ldots,\ell_{r}})$. Suppose $\mathcal H=(V,\mathcal E)$ is one extreme $r$-hypergraph which is free of $K^{(r)}_{\ell_1,\ldots,\ell_r}$ hypergraphs. We have 
\begin{equation} \label{eq:ubound_Kl..l-free_HGraph-2}
|\mathcal E|=e_{r}. 
\end{equation}
The neighbourhood of any vertex $v$ in $V$ is the collection of all $(r-1)$-subsets of  $V$ that form a $r$-hyperedge when combined with $v$: 
$$N(v)= \left \lbrace U \in  \binom{V}{r-1} \colon   \{v\} \cup U   \in \mathcal E \right \rbrace.$$ 
For any fixed $\ell_1$ vertices $v_1,\ldots,v_{\ell_1}$ let $\mathcal E'$ denote $N(v_1) \cap \cdots \cap N(v_{\ell_1}).$ The set 
$\mathcal E'$ can be considered as a collection of $(r-1)$-hyperedges. Let  $V' \subset V$ denote the vertices connected by $\mathcal E'$ thus forming a $(r-1)$-hypergraph $\mathcal H'=(V',\mathcal E')$ .  
Assume that  $\mathcal H'$ contains one $K^{(r-1)}_{\ell_2,\ldots,\ell_r}$ hypergraph, say $\mathcal{I}=(V(\mathcal{I}), \mathcal E(\mathcal{I}))$. Then the hypergraph 
$$\left(\{ v_1,\ldots, v_{\ell_1}\} \cup V(\mathcal{I})  , \left \{ \{ v_i \} \cup U   \mid U \in \mathcal E(\mathcal{I}), \: i=1,\ldots,\ell_1 \right \} \right)$$ 
would be $r$-uniform $K^{(r)}_{\ell_1,\ldots,\ell_r}$ and it would be included in $\mathcal{H}$, which is impossible. Hence  $(V',\mathcal E')$ must be $K^{(r-1)}_{\ell_2,\ldots,\ell_r}$-free and so we have 
\begin{equation} \label{eq:ubound_Kl..l-free_HGraph-3}
|\mathcal E'| = |N(v_1) \cap \cdots \cap N(v_{\ell_1})| \le e_{r-1} \quad \text{for any } v_1,\ldots,v_{\ell_1}.
\end{equation}

In order to use Theorem \ref{thm:overlapping-thm}, let us define the random variable $X \colon V \to \binom{V}{r-1}$ with uniform probability law $\Pr{X=U} =1/\binom{n}{r-1},$ for every $U \in \binom{V}{r-1}.$ For any $v \in V$  we define the event $E_v =  \{X \in N(v)\}.$  Note that $\sum_{v \in V} |N(v)|$  counts the number of all subsets of $r-1$ elements in every hyperedge of $V$, thus $\sum_{v \in V} |N(v)|= \binom{r}{r-1} \: |\mathcal E|$ and then we can write 
$$\sigma_1= \sum_{v \in V} \Pr{E_v}= \frac{1}{\binom{n}{r-1}} \sum_{v \in V} |N(v)|=\frac{r \: |\mathcal E|}{\binom{n}{r-1}}=\frac{r \: e_{r}}{\binom{n}{r-1}},$$ 
where in the last equality we have used \eqref{eq:ubound_Kl..l-free_HGraph-2}. 
Theorem \ref{thm:overlapping-thm} implies
\begin{equation*} \label{eq:ubound_Kl..l-free_HGraph-4}
\frac{1}{\binom{n}{r-1}} \sum_{\{ v_1,\ldots, v_{\ell_1} \}\in \binom{V}{\ell_1}} |N(v_1) \cap \cdots \cap N(v_{\ell_1})| \ge \binom{r \: e_{r}/\binom{n}{r-1}}{\ell_1}.
\end{equation*}
Using the inequality \eqref{eq:ubound_Kl..l-free_HGraph-3} we obtain
\begin{equation} \label{eq:ineq}
\frac{\binom{n}{\ell_1} \; e_{r-1}}{\binom{n}{r-1}} \ge \binom{r \; e_{r}/\binom{n}{r-1}}{\ell_1}.
\end{equation}
Now we estimate the left term in \eqref{eq:ineq}: 
$$
\frac{e_{r-1}\binom{n}{\ell_1}}{\binom{n}{r-1}} =  \frac{(r-1)!}{\ell_1 !} \frac{(n-(r-1))!}{(n-\ell_1)!} \: e_{r-1}.
$$
In the case $r-1 \le \ell_1$ we have
\[
\frac{e_{r-1}\binom{n}{\ell_1}}{\binom{n}{r-1}} =  \frac{(r-1)!}{\ell_1 !} \: e_{r-1}  (n-r+1)\cdots(n-\ell_1+1) \le \frac{(r-1)!}{\ell_1 !} \: e_{r-1} \: n^{\ell_1 - r+1}. 
\]
Otherwise we have $r -1 > \ell_1$ and then
\begin{align*}
\frac{e_{r-1}\binom{n}{\ell_1}}{\binom{n}{r-1}} &=  \frac{(r-1)!}{\ell_1 !} \: \frac{e_{r-1}}{(n-\ell_1) \cdots (n-r+2)} \le  \frac{(r-1)!}{\ell_1 !} \: \frac{e_{r-1}}{(n-r+2)^{r-1-\ell_1}}\\
&\le \frac{(r-1)!}{\ell_1 !} \: e_{r-1} \: (n-r+2)^{\ell_1 - r+1}.
\end{align*}
Hence we have 
$$
\frac{e_{r-1}\binom{n}{\ell_1}}{\binom{n}{r-1}} \le \frac{(r-1)!}{\ell_1 !} \: e_{r-1} \: n^{\ell_1 - r+1}(1+o(1)), \quad (n \to \infty). 
$$
A lower bound for the right term in \eqref{eq:ineq} is
\[
\binom{r \: e_{r}/\binom{n}{r-1}}{\ell_1} 
\ge \binom{\frac{ r! \: e_{r}}{ n^{r-1}}}{\ell_1} \ge \frac{\left(\frac{ r! \: e_{r}}{ n^{r-1}} -(\ell_1 -1) \right)^{\ell_1}}{\ell_1 !}.
\]
Combining the last two estimates and \eqref{eq:ineq} we can write 
\begin{equation} \label{eq:tbd-5}
\left( \frac{r! \: e_{r}}{n^{r-1}}-(\ell_1 -1)\right)^{\ell_1} \le (r-1)! \: e_{r-1} \: n^{\ell_1-r+1}(1+o(1)), \qquad (n \to \infty).
\end{equation}
By \eqref{bounds} we have that $e_r \gg n^{r-\frac{\ell_1 + \cdots + \ell_r - r}{\ell_1 \cdots \ell_r -1}}$. The fraction in this last exponent reaches its maximum for $r=\ell_1=\ell_2=2$, that is to say  
$
\frac{\ell_1 + \cdots + \ell_r - r}{\ell_1 \cdots \ell_r -1} \le  \frac{2}{3}.
$
Hence  $e_r / n^{r-1} \gg n^{1/3} \to \infty$, and then we have that $\ell_1 -1 = o\left( e_r / n^{r-1}\right) , \; (n \to \infty),$ which implies 
$$
\left( \frac{r! \: e_{r}}{n^{r-1}}-(\ell_1 -1)\right)^{\ell_1} = \left( \frac{r! \: e_{r}}{n^{r-1}}\right)^{\ell_1}(1+o(1)), \quad (n \to \infty).
$$
Using this last equality and \eqref{eq:tbd-5} we can write
$$
\left( \frac{r! \: e_{r}}{n^{r-1}}\right)^{\ell_1}(1+o(1)) \le (r-1)! \: e_{r-1} \: n^{\ell_1-r+1}(1+o(1)), \qquad (n \to \infty), 
$$
and so
\begin{align} 
e_{r} &\le   \frac{n^{r-1}}{r!} \big((r-1)!\big)^{1/\ell_1}\: (e_{r-1})^{1/\ell_1} \: n^{1-\frac{r-1}{\ell_1}} (1+o(1)) \notag \\
				& \le \frac{\big((r-1)!\big)^{1/\ell_1}}{r!}  \: (e_{r-1})^{1/\ell_1}\: n^{r-\frac{r-1}{\ell_1}}(1+o(1)), \quad (n \to \infty). \label{eq:tbd-4}
\end{align}
To prove the induction step assume that \eqref{ubound_Kl..l-free_HGraph} holds for $r-1$, then 
$$
e_{r-1} \le \frac{(\ell_{r}-1)^{1/\ell_2 \cdots \ell_{r-1}}}{(r-1)!} \; n^{r-1 - 1/\ell_2 \cdots \ell_{r-1}}(1+o(1)), \qquad (n \to \infty),
$$
and inserting this into \eqref{eq:tbd-4} we have
\begin{align*}
e_{r} &\le \frac{((r-1)!)^{1/\ell_{1}}}{r!}\: \left(\frac{(\ell_{r}-1)^{1/\ell_2 \cdots \ell_{r-1}}}{(r-1)!} \; n^{r-1 - 1/\ell_2 \cdots \ell_{r-1}}(1+o(1))\right)^{1/\ell_{1}} \: n^{r-\frac{r-1}{\ell_{1}}}\\
        &\le \frac{(\ell_r-1)^{1/\ell_1 \cdots \ell_{r-1}}}{r!} \: n^{r - 1/\ell_1 \cdots \ell_{r-1}} \: (1+o(1)), \qquad (n \to \infty).
\end{align*}
Hence \eqref{ubound_Kl..l-free_HGraph} also holds for $r$ and the proof of Theorem \ref{thm:ubound_Kl..l-free_HGraph} is completed.

\section{Infinite $\Ll$-free sequences of integers} \label{sec:Lfree-sequences}
\subsection{Proof of Theorem \ref{thm:liminf_Lfree_seq}}
We part the interval $(0,N^2]$ into the subintervals $I_j=((j-1)N,jN],\; j=1,\dots ,N$ and use the notation $A_j=A\cap I_j$. Note that $A(tN)=\sum_{j \le t} |A_j|, \; (1 \le t \le N).$ We will first estimate the sum
$$S=\sum_{j=1}^{N}\frac{|A_j|}{j^{1-1/(\ell_1\cdots \ell_{r-1})}}.$$
Let $\sigma(x)$ denote the number $\inf_{y>x}A(y)\dfrac{(y\log y)^{1/(\ell_1\cdots \ell_{r-1})}}y$.

On the one hand for any $t$ such that $1\le t\le N$ we have
$$A(tN)\ge \frac{\sigma(N)(tN)^{1-1/(\ell_1\dots \ell_{r-1})}}{(\log(tN))^{1/(\ell_1\dots \ell_{r-1})}}\ge \frac{\sigma(N)(tN)^{1-1/(\ell_1\dots \ell_{r-1})}}{(2\log N)^{1/(\ell_1\dots \ell_{r-1})}}.$$
We use this inequality and summation by parts to get
\begin{eqnarray}\label{S1}S&\ge &\left (1-1/(\ell_1\cdots \ell_{r-1})\right )\int_1^{N}\frac{\sum_{j\le t}|A_j|}{t^{2-1/(\ell_1\cdots \ell_{r-1})}} \ud t \nonumber\\
&\ge &\frac 12\int_1^{N}\frac{A(tN)}{t^{2-1/(\ell_1\cdots \ell_{r-1})}} \ud t \nonumber\\
&\ge &\frac{\sigma(N)N^{1-1/(\ell_1\cdots \ell_{r-1})} }{4(\log N)^{1/(\ell_1\dots \ell_{r-1})}} \int_1^{N}\frac{\ud t}{t}\nonumber\\ &\ge &\frac{\sigma(N)}4\left (N\log N \right )^{1-1/(\ell_1\cdots \ell_{r-1})}.
\end{eqnarray}
On the other hand H\"{o}lder inequality yields
$$S\le \left (\sum_{j=1}^{N}|A_j|^{\ell_1\cdots \ell_{r-1}}\right )^{1/(\ell_1\cdots \ell_{r-1})}\left (\sum_{j=1}^{N}\frac 1j\right )^{1-1/(\ell_1\cdots \ell_{r-1})}.$$
At this point we will need the following result.
\begin{lem}\label{sum}
Let $A\subset \Z$ be $\mathcal L^{(r)}_{\ell_1,\dots,\ell_r}$-free and $A_j=A\cap ((j-1)N,jN],\ j=1,\dots, N$. Then we have
$$\sum_{j\le N}|A_j|^{\ell_1\cdots \ell_{r-1}}\ll N^{\ell_1\cdots \ell_{r-1}-1}.$$
\end{lem}
Assuming Lemma \ref{sum} holds (we will prove it below) we can write
\begin{equation}\label{S2}S\ll N^{1-1/(\ell_1\cdots \ell_{r-1})}(\log N)^{1-1/(\ell_1\cdots \ell_{r-1})}.\end{equation}
Inequalities \eqref{S1} and \eqref{S2} imply $\sigma(N)\ll 1$  that is precisely the claim of Theorem \ref{thm:liminf_Lfree_seq}.
\begin{proof}[Proof of Lemma \ref{sum}] We use induction on $r$. We will call $\ell_1$-subsets to the subsets of $\ell_1$ elements.

When $r=2$ then  $\sum_{j\le N}\binom{|A_j|}{\ell_1}$ counts the $\ell_1$-subsets in $A$ included in some of the intervals $I_j=((j-1)N,jN],\; j=1,\dots ,N.$ 
We will estimate this number in two steps. On the one hand there are $\binom{N-1}{\ell_1-1}$ classes of pairwise congruent $\ell_1$-subsets of $A$ with diameter less than $N$. The reason is that each of these classes contains a representative subset that is within $[1,N]$ and which contains $1$, note that the remaining elements of the representative subset can be chosen in $\binom{N-1}{\ell_1-1}$ different ways.  On the other hand it is easy to see that since $A$ is $\Ll_{\ell_1,\ell_2}$-free then necessarily every class of pairwise congruent $\ell_1$-subsets contains at most $\ell_2$ members. Hence we have
$$ \sum_{j\le N} |A_j|^{\ell_1} \ll \sum_{j\le N}\binom{|A_j|}{\ell_1} \le \ell_2 \binom{N-1}{\ell_1-1} \ll N^{\ell_1-1}.  $$

When $r \ge 3$ assume that Lemma \ref{sum} is true for $r-1$.
For any set $S$ and any collection $x=\{x_1,\dots,x_{\ell_1}\}\in \binom N{\ell_1}$ we will use the notation $S*x=\bigcap_{i=1}^{\ell_1}(S+x_i).$

On the one hand H\"{o}lder inequality yields
\begin{equation}\label{e1}\sum_{x\in \binom N{l_{1}}}|A_j*x|\le \left (\sum_{x\in \binom N{l_{1}}}|A_j*x|^{\ell_2\cdots \ell_{r-1}}\right )^{1/(\ell_2\cdots \ell_{r-1})}
\binom N{\ell_1}^{1-1/(\ell_2\cdots \ell_{r-1})}.
\end{equation}
On the other hand Lemma \ref{lem:over} with $X=[2N]$ and $B=[N]$ implies
\begin{equation}\label{e2}\sum_{x\in \binom N{\ell_1}}|A_j*x|\ge 2N\binom{|A_j|/2}{\ell_1}\gg N|A_j|^{\ell_1}.\end{equation}
Combining \eqref{e1} and \eqref{e2} we obtain
$$N^{\ell_2\cdots \ell_{r-1}} |A_j|^{\ell_1\cdots \ell_{r-1}}   \ll \left (\sum_{x\in \binom N{\ell_1}}|A_j * x|^{\ell_2\cdots \ell_{r-1}}\right )
N^{\ell_1\cdots \ell_{r-1}-\ell_1} $$
Summing in $j$ we can write:
$$N^{\ell_2\cdots \ell_{r-1}}\sum_{j \le N} |A_j|^{\ell_1\cdots \ell_{r-1}}   \ll N^{\ell_1\cdots \ell_{r-1}-\ell_1}\sum_{x\in \binom N{\ell_1}}\sum_{j \le N}|A_j*x|^{\ell_2\cdots \ell_{r-1}}
$$
Observe that $A_j$ is $\Ll_{\ell_1,\ldots,\ell_r}^{(r)}$-free (because $A_j \subset A$) and then Lemma \ref{lem:TBD-1} implies that for any $x=\{x_1,\dots,x_{\ell_1}\}\in \binom N{\ell_1}$ the set $A_j*x=\bigcap_{i=1}^{\ell_1}(A_j+x_i)$ is $\Ll^{(r-1)}_{\ell_2,\dots,\ell_r}$-free.
Hence we apply the induction hypothesis to each $A_j*x$ to obtain
$$N^{\ell_2\cdots \ell_{r-1}}\sum_{j \le N} |A_j|^{\ell_1\cdots \ell_{r-1}}   \ll N^{\ell_1\cdots \ell_{r-1}-\ell_1}\sum_{x\in \binom N{\ell_1}}N^{\ell_2\cdots \ell_{r-1}-1}
$$
Thus we have
$$\sum_{j \le N} |A_j|^{\ell_1\cdots \ell_{r-1}}   \ll N^{\ell_1\cdots \ell_{r-1}-1},
$$
as claimed.
\end{proof}
\subsection{Proofs of Theorems \ref{thm:2L-free_seq} and \ref{thm:L2...2-free_seq} }
The strategy of the proof is the same for the two cases of Hilbert cubes and $\Ll_{2,\ell}$. 
We first construct a dense random sequence $S$ free of arithmetic progressions. 
We will say that $X$ is an \emph{obstruction} (for $S$) when $X \subset S$ is a sumset of the class $\Ll_{\ell_1,\ldots,\ell_r} ^{(r)}.$
The sequence $S$ is likely to have infinitely many obstructions. If we could proof that obstructions are few then we would be able to remove all of them by just removing few elements from $S$. After the removal process we would retain a subsequence $A \subset S$ satisfying the conditions of Theorem \ref{thm:2L-free_seq} (resp. Theorem \ref{thm:L2...2-free_seq}).
Thus we have to estimate the number of  obstructions for $S$.  In the cases of Hilbert cubes and $\Ll_{2,\ell}$ we have succeeded to obtain an upper bound which allows to complete the proofs of Theorems \ref{thm:L2...2-free_seq} and \ref{thm:2L-free_seq}. 

Our first remark is that we can take $\epsilon$ as little as needed in the sense that if Theorem \ref{thm:2L-free_seq} is true for a particular $\epsilon_0 > 0$ then it is also true for any $\epsilon > \epsilon_0$. 

We define a collection of  intervals as follows: $$I_m=[4^{m+2},4^{m+2}+4^m), \quad (m \ge 1).$$ 
Let $B_m$  denote the set given by Theorem \ref{Behrend} of Behrend 
 with the following properties: $B_m\subset I_m$, $B_m$ is free of arithmetic progressions and has size $|B_m|\ge 4^{m(1+o(1))}.$ 

Given $\epsilon >0$ we have $|B_m|\ge 4^{m(1-\epsilon/2)}, \, (m \ge m_\epsilon),$ for some positive integer $m_\epsilon.$
We take $$B=\bigcup_{m \ge m_\epsilon} B_m.$$
Let $r, \ell_1,\ldots,\ell_r$ be integers with $r\ge 2$ and $2\le \ell_1\le \cdots \le \ell_r$.
We consider the probability space of all random infinite sequences $S$ of positive integers with law $$\mathbb P(\nu\in S)=\begin{cases} f(\nu) \text{ if } \nu\in B,\\ 0 \text{ otherwise},\end{cases}$$ where all the events $\{\nu\in S\}_{\nu \ge 1}$ are mutually independent and 
$$f(\nu)=\nu^{-\alpha} ,\qquad  \alpha=\frac{\ell_1+\cdots+\ell_r-r}{\ell_1 \cdots \ell_r-1}+\epsilon/2.$$
We will write $S_m$ for the intersection $S\cap B_m$.
\begin{lem} \label{lem:tbd-2}
Any random sequence $S$ defined as above is free of arithmetic progressions and with high probability we have
\begin{equation} \label{lbound:|S_m|}
|S_m| \gg 4^{m\left(1-\frac{\ell_1+\cdots+\ell_r-r}{\ell_1 \cdots \ell_r-1} - \epsilon \right)}, \quad (m \to \infty).
\end{equation}
\end{lem}
\begin{proof}As $S \subset B,$ it suffices to proof that the set $B$  does not contain arithmetic progressions.
Take any $x_1 < x_2 < x_3$ with $x_1\in B_{m_1},\ x_2\in B_{m_2},\ x_3\in B_{m_3}$ and $m_1\le m_2 \le m_3$. 
If $m_1=m_2=m_3$ then $x_1, x_2, x_3$ are not in arithmetic progression because $B_{m_1}$ is free of them. 

If $m_1<m_2<m_3$ and $2x_2=x_1+x_3$ then we would have
$$34 \cdot 4^{m_2} = 2(4^{m_2+2}+4^{m_2}) \ge 2 x_2 = x_1+x_3 \ge 4^{m_3+2} \ge 4^{m_2+3}=64 \cdot 4^{m_2},$$
which is false.
If $m_1<m_2=m_3$ then $$x_2-x_1 \ge 4^{m_2+2}-(4^{m_2+1}+4^{m_2-1}) = \frac{47}{4} 4^{m_2} > x_3-x_2,$$ 
and then $x_1, x_2, x_3$ are not in arithmetic progression.
If $m_1=m_2<m_3$ then  $$x_3-x_2 \ge 4^{m_2+3}-(4^{m_2+2}+4^{m_2}) = 47 \cdot 4^{m_2}  >  x_2-x_1,$$
and then $x_1, x_2, x_3$ are not in arithmetic progression. 

When $\nu \in B_m$ then $\nu < 4^{m+3}$  and so the expected size of $S_m=S\cap B_m$ is
$$\mu_m=\mathbb E(|S_m|)=\sum_{\nu\in B_m}\nu^{-\alpha}\ge |B_m|4^{-(m+3)\alpha} \gg  4^{m(1-\epsilon/2- \alpha)}, \quad (m \to \infty).$$ 
Since $|S_m|$ is a sum of mutually independent indicator random variables we can apply Chernoff inequality to obtain
$$\mathbb P(|S_m|<\mu_m/2)< e^{-\mu_m/2} \ll e^{- 4^{m(1-\epsilon/2-\alpha)}}.$$
This inequality implies $\sum_{m \ge 3} \mathbb P(|S_m|<\mu_m/2)<\infty$, and then by the Borell-Cantelli Lemma we have that with high probability  
\begin{equation*} 
|S_m|\ge \mu_m/2 \gg 4^{m(1-\epsilon/2-\alpha)} = 4^{m(1-\frac{\ell_1+\cdots+\ell_r-r}{\ell_1 \cdots \ell_r-1}-\epsilon)}, \quad (m \to \infty).
\end{equation*}

\end{proof}
We want to prune the sequence $S$ in order to destroy all obstructions. To this end we will remove from every obstruction its greatest element. Let  $S^{\mathrm{bad}}_m$ denote the collection of all elements in $S_m$ that have the property to be the greatest element in at least one obstruction (to $S$):
$$S^{\mathrm{bad}}_m := \left \lbrace s \in S_m \mid s=\max(X),\ X \subset S \text{ is an obstruction} \right \rbrace.$$ We also define a random variable that counts the number of obstructions that have their maximum in $S_m$:
$$
N(S_m) := \left \lbrace  X \subset S \text{ is an obstruction } \mid \max(X) \in S_m \right \rbrace.
$$
For two particular cases we claim that obstructions with a maximum in $S_m$ are few.
\begin{lem} \label{lem:obstructions-are-few} For the two cases  $\Ll_{2,\ell}$ and $\Ll_{2,\ldots,2}^{(r)}$ we have with high probability
\begin{equation} \label{eq:sumsets-are-few}
|N(S_m)| = o(|S_m|),  \quad (m \to \infty).
\end{equation}
\end{lem}
We postpone the proof of Lemma \ref{lem:obstructions-are-few} until the end of this section.

Now we can end the proof of Theorems \ref{thm:2L-free_seq}  and  \ref{thm:L2...2-free_seq} as follows.
Take the randomly constructed sequence $S$. For every obstruction $X \subset S$ we have that $\max(X) \in S_m$ for some $m$. We remove $\max(X)$ from the set $S_m$. We perform this removal process for all the obstructions for $S$. Let $S_m^*$ denote the subset of $S_m$ that is retained after the completion of this removal process. Two different obstructions might have the same maximum therefore $N(S_m) \ge |S^{\mathrm{bad}}_m|$. Thus by Lemma \ref{lem:tbd-2} and Lemma \ref{lem:obstructions-are-few}, with high probability we have that the retained elements are at least
\begin{align*}
|S_m^*|=|S_m \setminus S^{\mathrm{bad}}_m| \ge |S_m|-|N(S_m)|&\gg |S_m| \gg 4^{m(1-\frac{\ell_1+\cdots+\ell_r-r}{\ell_1 \cdots \ell_r-1}-\epsilon)}, \quad (m \to \infty),
\end{align*} 
for the two cases $\Ll_{2,\ell}$ and $\Ll_{2,\ldots,2}^{(r)}$.

Finally let us take $A = \bigcup_{m\ge m_\epsilon} S_m^*$. On the one hand $A$ is $\Ll_{\ell_1,\ldots,\ell_r}^{(r)}$-free because all sumsets of the class $\Ll_{\ell_1,\ldots,\ell_r}^{(r)}$ that were contained in the initial sequence $S$ have been destroyed in the process of obtaining the subsequence $A \subset S$.

On the other hand for each $x$ large enough take the integer $k$ such that $4^k<x\le 4^{k+1}$. It is clear that 
$$A(x)\ge \sum_{m\le k}|S_m^*| \ge |S_k^*|\gg 4^{k(1-\frac{\ell_1+\cdots+\ell_r-r}{\ell_1 \cdots \ell_r-1}-\epsilon)}\gg x^{1-\frac{\ell_1+\cdots+\ell_r-r}{\ell_1 \cdots \ell_r-1}-\epsilon}, \quad (x \to \infty),$$
holds, with high probability,  for the two cases $\Ll_{2,\ell}$  and $\Ll_{2,\ldots,2}^{(r)}$. Therefore at least one sequence must exist satisfying Theorem \ref{thm:2L-free_seq}. The same applies for Theorem \ref{thm:L2...2-free_seq}.

\

In the proof of Lemma \ref{lem:obstructions-are-few} we will use the following well known estimates:
\begin{equation} \label{eq:sum-1} 
\sum_{n\le x} n^{\beta}\asymp x^{1+\beta} \quad \text{ if } \beta >-1.
\end{equation}
\begin{equation} \label{eq:sum-2}
\sum_{x\le n} n^{\beta}\asymp x^{1+\beta} \quad \text{ if } \beta <-1.
\end{equation}
%
%
\subsubsection{Proof of Lemma \ref{lem:obstructions-are-few} for the case $\Ll_{2,\ell}$} Any sumset $X$ in $\Ll_{2,\ell}$ can be described as follows:
$$X = \{0,x\} + \{y_1, \ldots, y_\ell \}, \qquad y_1 \le \cdots \le y_\ell.$$ 
%
%
For any fixed choice of $x,y_1, \ldots, y_\ell$ either (a) there exists a $t$ with  $2 \le t \le  \ell$ such that 
\begin{equation} \label{eq:TBC2}
y_1 \le \cdots \le y_{t-1} \le x \le y_t \le \cdots \le y_\ell,
\end{equation}
or alternatively (b) we have  either  $x \le y_1 \le \cdots \le y_\ell$  (type \lq \lq left\rq \rq) or $ y_1 \le \cdots \le y_\ell \le x$  (type \lq \lq right\rq \rq).
For convenience we will say that $X$ is \lq\lq of type $t$\rq \rq \  when $X = \{0,x\} + \{y_1, \ldots, y_\ell \}$ and \eqref{eq:TBC2} holds.

Suppose that a sumset $X$ of the class $\Ll_{2,\ell}$ is contained in $S$. By Lemma \ref{lem:tbd-2}, 
$S$ does not contain arithmetic progressions, hence $X$ is also free of them. Then Lemma \ref{lem:degen-sumset-contains-AP} implies that $X$ cannot be degenerated, that is to say all the possible sums that contribute to the sumset $X$ are distinct (see Definition \ref{degenerated-sumset}). Hence
$$P(X \subset S) = f(y_1)\cdots f(y_\ell)f(x+y_1)\cdots f(x+y_\ell).$$
We will estimate first the expected number of obstructions $X$ (to $S$) such that $\max(X) \in S_m$ for the cases left and right. 

Let $N_m^L$ denote the number of number of such obstructions that satisfy $x \le y_1 \le \cdots \le y_\ell$. 
The function $f(\nu)=\nu^{-\alpha}$, where $\alpha=\frac{l}{2l-1}+\epsilon/2
$, is non-increasing so $f(x+y_i) \le f(y_i)$, and we can write
$$
P(X \subset S) = f(y_1)\cdots f(y_\ell)f(x+y_1)\cdots f(x+y_\ell) \le f(y_1)^2\cdots f(y_\ell)^2.
$$
Note also that if $x+y_\ell=\max(X) \in S_m \subset [4^{m+2},4^{m+2}+4^m)$ then $y_\ell \le 4^{m+3}.$
Hence we can write
\begin{align*}
\Esp{N_m^L} &= \sum_{\substack{X \text{ of type left end}  \\ \max(X) \in S_m}} \Pr{ X \subset S} \\ 
&\le \sum_{x \le y_1 \le \cdots y_\ell \le 4^{m+3}}  f(y_1)^2\cdots f(y_\ell)^2 = \sum_{x \le y_1 \le \cdots y_\ell \le 4^{m+3}}  y_{1}^{-2\alpha} \cdots y_{\ell}^{-2\alpha} \\
& \le \sum_{x \le 4^{m+3}} \left( \sum_{ x \le y} y^{-2\alpha}\right)^{\ell} \ll \sum_{x \le 4^{m+3}}  x^{(1-2\alpha) \ell} \ll 4^{m(1+(1-2\alpha)\ell)}, 
\end{align*}
where, since $2\alpha=\frac{2l}{2l-1}+\epsilon >1$ and $2\alpha-1=\frac 1{2l-1}+\epsilon<1$ for $\epsilon$ small enough, we have used the estimates \eqref{eq:sum-1} and \eqref{eq:sum-2}. 

Let $N_m^R$ denote the number number of obstructions $X$ (to $S$) such that $\max(X) \in S_m$, with $y_1 \le \cdots \le y_\ell \le x$. Again by the monotony of $f$ we have
$$
P(X \subset S) = f(y_1)\cdots f(y_\ell)f(x+y_1)\cdots f(x+y_\ell) \le f(y_1)\cdots f(y_\ell)f(x)^\ell,
$$
and then
\begin{align*}
\Esp{N_m^R} &= \sum_{\substack{X \text{ of type right end}  \\ \max(X) \in S_m}} \Pr{ X \subset S} \\ 
&\le \sum_{y_1 \le \cdots y_\ell \le x \le 4^{m+3}}  f(y_1) \cdots f(y_\ell) f(x)^\ell = \sum_{y_1 \le \cdots y_\ell \le x \le 4^{m+3}}  y_{1}^{-\alpha} \cdots y_{\ell}^{-\alpha} x^{-\alpha \ell} \\
& \le \sum_{x \le 4^{m+3}} \left( \sum_{ y \le x} y^{-\alpha}\right)^{\ell} x^{-\alpha \ell} \ll \sum_{x \le 4^{m+3}}  x^{(1-2\alpha) \ell} \ll 4^{m(1+(1-2\alpha)\ell)}, 
\end{align*}
where we have taken $\epsilon$ sufficiently small to have $ \alpha < 1$.

Fix $t$ with $2 \le t \le \ell$. 
By \eqref{eq:TBC2} and the monotony of the function $f$ we have:
$$ f(x+y_i) \le f(x), \; (1\le i \le t-1), \qquad  f(x+y_i) \le f(y_i) , \; (t\le i \le \ell).$$ 
Then we can write
\begin{align*}
\Pr{X \subset S}	&\le f(y_1)\cdots f(y_\ell) f(x)^{t-1} f(y_t)\cdots f(y_\ell) \\
								& =f(y_1)\cdots f(y_{t-1}) f(x)^{t-1} f(y_t)^2\cdots f(y_\ell)^2.
\end{align*}
Let $N_m^{(t)}=N_m^{(t)}(S)$ denote the number of sumsets $X$ of type $t$ such that $X\subset S$ and  $\max(X) \in S_m$.  
The expected value of $N_m^{(t)}$ can be estimated as follows 
\begin{align*}
\Esp{N_m^{(t)}} &= \sum_{\substack{X \text{ of type } t \\ \max(X) \in S_m}} \Pr{ X \subset S} \\ 
&\le \sum_{\substack{ y_1 \le \cdots \le y_{t-1} \le x \\ x \le y_t \le \cdots y_\ell \le 4^{m+3}}} f(y_1)\cdots f(y_{t-1}) f(x)^{t-1} f(y_t)^2\cdots f(y_\ell)^2 \\
&\le \sum_{\substack{ y_1, \ldots, y_{t-1} \le x \\ x \le y_t \le \cdots y_\ell \le 4^{m+3}}} y_1^{-\alpha} \cdots y_{t-1}^{-\alpha} x^{-(t-1)\alpha} y_{t}^{-2\alpha} \cdots y_{\ell}^{-2\alpha} \\
& \le \sum_{x \le y_t \le \cdots y_\ell \le 4^{m+3}} \left( \sum_{ y \le x} y^{-\alpha}\right)^{t-1} x^{-(t-1)\alpha} y_{t}^{-2\alpha} \cdots y_{\ell}^{-2\alpha}\\ 
& \le  \sum_{x\le 4^{m+3}}x^{-(t-1)\alpha} \left( \sum_{x \le y } y^{-2\alpha}\right)^{\ell-t+1} \left( \sum_{ y \le x} y^{-\alpha}\right)^{t-1}\\ 
\tiny{
(\text{since }  \tfrac 12 < \alpha < 1 \text{ for } \epsilon \text{ small})
}
& \le  \sum_{x\le 4^{m+3}}x^{-(t-1)\alpha}x^{(1-2\alpha)(\ell-t+1)} x^{(1-\alpha)(t-1)}\\  & \le  \sum_{x\le 4^{m+3}}x^{-(2\alpha-1)\ell}\ll 4^{m(1-(2\alpha -1)\ell)},
\end{align*}
since $(2\alpha-1)\ell=\frac{\ell}{2\ell-1}+\epsilon \ell<1$ for $\epsilon$ small enough.
Observe that $N(S_m)=N_m^L+\sum_{t=2}^{\ell} N_m^{(t)} +N_m^R$, hence we can write
\begin{align*}
\Esp{N(S_m)} = \Esp{N_m^L} + \sum_{t=2}^{\ell} \Esp{N_m^{(t)}}  + \Esp{N_m^R} \ll  4^{m(1+(1-2\alpha)\ell)}. 
\end{align*}
Markov inequality yields
\begin{eqnarray*}
\sum_{m\ge m_\epsilon} \mathbb P\left (N(S_m) >m^2\mathbb E(N(S_m))\right )\le \sum_{m \ge m_\epsilon} \frac 1{m^2}<\infty.
\end{eqnarray*}
Thus by the Borell-Cantelli Lemma with high probability we have 
\begin{equation*}
N(S_m) \ll m^2 \Esp{N(S_m)} \ll 4^{m(1+(1-2\alpha)\ell+o(1))} = 4^{m( 1 -\frac{\ell}{2\ell-1} - \ell \epsilon+o(1))}.
\end{equation*}
Using this and the estimate \eqref{lbound:|S_m|} we have with high probability 
$$ \frac{N(S_m)}{|S_m|} \ll 4^{m((1-\ell)\epsilon+o(1))} \to 0, \quad (m \to \infty),$$
since $\ell \ge 2$, which proves Lemma \ref{lem:obstructions-are-few} for the case $\Ll_{2,\ell}$.

\subsubsection{Proof of Lemma \ref{lem:obstructions-are-few} for the case of Hilbert cubes}
Any Hilbert cube $X$ of dimension $r$ can be written as $$X=x_0+\{0,x_1\}+\cdots +\{0,x_r\},\quad (x_1\le \cdots \le x_r).$$ 
Indeed if $X=L_1+\cdots+L_r$ with $L_j=\{a_j,b_j\}$, take $x_0=\sum_{j=1}^{r} a_j$ and $x_j=b_j-a_j,$ rearranging the indexes if needed to have $x_1\le \cdots \le x_r.$ In other words, $X=\{x_0+\sum_{i \in I} x_i \mid I \subset [r]\}$, where the indexes in the sum cover all subsets of the interval $[r]=\{1,\ldots,r\}.$

Suppose that a Hilbert cube $X$ is contained in $S$. By Lemma \ref{lem:tbd-2}, 
$S$ does not contain arithmetic progressions, hence $X$ is also free of them. Then Lemma \ref{lem:degen-sumset-contains-AP} implies that $X$ cannot be degenerated, that is to say all the possible sums that contribute to the sumset $X$ are distinct (see Definition \ref{degenerated-sumset}).

For any fixed choice of $x_0, x_1, \ldots, x_r $ with $x_1\le \cdots \le x_r $, the probability that the 
corresponding Hilbert cube $X=\{x_0+\sum_{i \in I} x_i \mid I \subset [r]\}$ is contained in the random infinite sequence $S$ is
\begin{align*}
\Pr{X \subset S} &= \Pr{\bigwedge_{I\subset [r]} (x_0+\sum_{i\in I}x_i \in S)}= \prod_{I\subset [r]}\Pr{x_0+\sum_{i\in I}x_i \in S}\\
&=\prod_{I\subset [r]} \left (x_0+\sum_{i\in I}x_i \right)^{-\alpha},
\end{align*}
because all the sums $x_0+\sum_{i\in I}x_i $ are distinct.  The indexes $I$ in the sum (and in the product) cover all subsets of the interval $[r]$, that is to say: $\emptyset$, and -for each $i=1,\ldots,r$- all the $2^{i-1}$ subsets of $[r]$ in which $i$ is the maximum. 
Note that having fixed $i$, for each $s$-uple $k_1, \cdots, k_s$ with $1 \le k_1 \le \cdots \le k_s \le i$ we have $$(x_0+x_{k_1}+\cdots+x_{k_s}+x_i)^{-\alpha} \le (x_0+x_i)^{-\alpha}.$$  Hence we can write
\begin{align*}
\Pr{X \subset S} &\le x_0^{-\alpha} \prod_{i=1}^r\prod_{\substack{I\subset [r],\\ \max I=i}} (x_0+x_i)^{-\alpha}\\
&\le x_0^{-\alpha} \prod_{i=1}^r  (x_0+x_i)^{-2^{i-1}\alpha}.
\end{align*}
It is convenient to write  $y_0=x_0$ and $y_i=x_0+x_i$. With this notation we have
\begin{equation} \label{eq:TBC1}
\Pr{X \subset S} \le  y_0^{-\alpha} \prod_{i=1}^r y_i^{-2^{i-1}\alpha}.
\end{equation} 
Note that if $y_r+x_1+\cdots+x_{r-1} = \max(X) \in S_m \subset [4^{m+2}, 4^{m+2}+4^m)$ then necessarily $y_r \le 4^{m+3}$. Hence by \eqref{eq:TBC1}
\begin{align*}
\mathbb E(N(S_m)) &\le \sum_{\substack{ X=\{x_0+\sum_{i \in I} x_i \mid I \subset [r]\} \\ x_1\le \cdots \le x_r, \; \max(X) \in S_m }} \Pr{X \subset S}\\ 
& \le\sum_{y_0\le y_1\le\cdots \le y_r\le 4^{m+3}} y_0^{-\alpha} \prod_{i=1}^r y_i^{-2^{i-1}\alpha} .
\end{align*}  Taking in account that $\alpha=\frac r{2^r-1}+\epsilon/2$ and that $t-2^{t}\alpha>-1$ for all $t\le r$ and $\epsilon $ small enough we can estimate $\mathbb E(N(S_m))$ as follows:
\begin{align*}
\mathbb E(N(S_m))&\ll \sum_{y_0\le\cdots \le y_r \le 4^{m+3}} y_0^{-\alpha}y_1^{-\alpha} \prod_{i=2}^r y_i^{-2^{i-1}\alpha}\\ &\ll \sum_{y_1\le\cdots \le y_r \le 4^{m+3}} y_1^{1-2\alpha} \cdot y_2^{-2\alpha} \prod_{i=3}^r y_i^{-2^{i-1}\alpha}\\&\ll \cdots \\
&\ll \sum_{y_t\le\cdots \le y_r \le 4^{m+3}} y_t^{t-2^{t}\alpha} \cdot y_{t+1}^{-2^t\alpha} \prod_{i=t+2}^r y_i^{-2^{i-1}\alpha}\\
&\ll \cdots \\
&\ll \sum_{ y_{r-1} \le y_r \le 4^{m+3}} y_{r-1}^{r-1-2^{r-1}\alpha} \cdot y_{r}^{-2^{r-1}\alpha}  \\
&\ll \sum_{y_r\le 4^{m+3}} y_r^{r-2^r\alpha}\\
&\ll (4^m)^{1+r-2^r\alpha}.
\end{align*}
Markov inequality yields
\begin{eqnarray*}
\sum_{m\ge m_\epsilon} \mathbb P \Big( N(S_m) >m^2\mathbb E(N(S_m)) \Big)\le \sum_{m \ge m_\epsilon} \frac 1{m^2}<\infty.
\end{eqnarray*}
Thus by the Borell-Cantelli Lemma with high probability we have 
\begin{align}
N(S_m)	&\ll m^2\Esp{N(S_m)} \ll 2^{m(1+r-2^r\alpha + o(1))} \notag \\
	&\ll 2^{m({1-\frac r{2^r-1}-2^{r-1}\epsilon} +o(1))}, \quad (m \to \infty) \label{ubound:N_m(S)}.
\end{align}
Combining \eqref{ubound:N_m(S)} with \eqref{lbound:|S_m|}  we have with high probability
$$ \frac{N(S_m)}{|S_m|} \ll 2^{m((1-2^{r-1})\epsilon + o(1))} \to 0, \quad (m \to \infty),$$
 which proves Lemma \ref{lem:obstructions-are-few} for the $\Ll_{2,\ldots,2}^{(r)}$ case (Hilbert cubes).

\todos

\end{document}